\documentclass[preprint, 12pt]{amsart} 

\usepackage[utf8]{inputenc} 
\usepackage{amsmath}	
\usepackage{amssymb}	
\usepackage{amsthm}
\usepackage{enumerate}
\usepackage{bbm}

\usepackage[includeall,
            vmargin={1in,1in},
            hmargin={1in,0.375in}
            ]{geometry}

\newtheorem{thm}{Theorem}
\newtheorem*{thm1}{Theorem \ref{thm:countable}}
\newtheorem*{thmrat}{Theorem \ref{thm:ratio}}

\newtheorem{prop}[thm]{Proposition}
\newtheorem{cor}[thm]{Corollary}

\theoremstyle{definition}
\newtheorem{defi}{Definition}

\newtheorem{example}{Example}

\newcommand{\prob}{\mathbb{P}}

\def\ZZ{{\mathbb Z}}
\def\NN{{\mathbb N}}
\def\PP{{\mathbb P}}

\newcommand{\bracks}[1]{\left( #1 \right)}

\newcommand{\var}{\operatorname{var}}

\begin{document}

\title{Random-step Markov processes}

\author[N. Bushaw]{Neal Bushaw}
\address{School of Mathematical and Statistical Sciences,
Arizona State University,
P.O.\,Box 871804,
Tempe, AZ 85287-1804,
USA,
and
IMPA, 
Estrada Dona Castorina 110,
Jardim Bot\^anico,
Rio de Janeiro, RJ, 
Brasil}
\email{neal@asu.edu}

\author[K. Gunderson]{Karen Gunderson}
\address{Heilbronn Institute for Mathematical Research,
School of Mathematics,
University of Bristol, 
Bristol BS8 1TW,
United Kingdom}
\email{karen.gunderson@bristol.ac.uk}

\author[S. Kalikow]{Steven Kalikow} 
\address{
Department of Mathematics,
University of Memphis,
3725 Norriswood,
Memphis, TN 38152, 
USA}
\email{skalikow@memphis.edu}

\date{4 October, 2014}


\begin{abstract}
In this paper, we explore two notions of stationary processes.  The first is called a \emph{random-step Markov process} in which the stationary process of states, $\left(X_i\right)_{i \in \mathbb{Z}}$ has a stationary coupling with an independent process on the positive integers, $\left(L_i\right)_{i \in \mathbb{Z}}$ of `random look-back distances'.  That is,
 $L_0$ is independent of the `past states', $\left(X_i, L_i\right)_{i<0}$, and  for every positive integer $n$, the probability distribution on the `present', $X_0$, conditioned on the event $\left\{L_0 = n\right\}$  and on the past is the same as the probability distribution on $X_0$ conditioned on the `$n$-past', $\left(X_i\right)_{-n\leq i <0}$ and $\{L_0 = n\}$.  A random Markov process is a generalization of a Markov chain of order $n$ and has the property that the distribution on the present given the past can be uniformly approximated given the $n$-past, for $n$ sufficiently large.  Processes with the latter property are called uniform martingales, closely related to the notion of a 
`continuous $g$-function'.

In this paper, it is shown that every stationary process on a countable alphabet that is a uniform martingale and is dominated by a finite measure on that alphabet is also a random Markov process and that the random variables $\left(L_i\right)_{i \in \mathbb{Z}}$ and associated coupling can be chosen so that the distribution on the present given the $n$-past and the event $\left\{L_0 = n\right\}$ is `deterministic': all probabilities are in $\left\{0,1\right\}$.  In the case of finite alphabets, those random-step Markov processes for which $L_0$ can be chosen with finite expected value are characterized.  For stationary processes on an uncountable alphabet, a stronger condition is also considered which is sufficient to imply that a process is a random Markov processes.  In addition, a number of examples are given throughout to show the sharpness of the results.
\end{abstract}

\maketitle

\section{Introduction}
\label{sec:intro}

\subsection{Definitions and results}\label{subsec:def}

A random-step Markov process, introduced by Kalikow in~\cite{Kalikow90} under the name `random Markov process', is a natural generalization of the classical $n$-step Markov chain. In this type of stationary process, rather than the current state determining the probability distribution of the state in the next time step, as in a Markov chain, there is a random `look-back time' which determines how much of the state history is used to find the distribution of the random variable at the next time step. If the random look-back time is bounded above by some $n \in \NN$, then the random Markov process is trivially an $n$-step Markov chain.  In this paper, we present new results on characterizations those of stationary processes on both countable and uncountable alphabets that are random-step Markov processes as well as a number of examples that show the sharpness of these results.

The following notation is used throughout the paper:  For any set $A$ and $\Omega = A^{\mathbb{Z}}$, the set $A$ is called the \emph{alphabet} for $\Omega$ and $\Omega$ is represented as the set of doubly infinite words on $A$.  For a given word $\mathbf{\omega} = (\omega_i)_{i \in \mathbb{Z}} \in\Omega$, the infinite sequence $\bracks{\omega_i}_{i\in\ZZ_{-}}\in A^{\ZZ_{-}}$ is called the \emph{past}. Similarly, for any $m \geq 1$, the sequence  $\bracks{\omega_{i}}_{i=-m}^{-1}$ is called the \emph{$m$-past} of $\mathbf{\omega}$. To condense notation, we will write, for example, $\PP\bracks{X_0=\omega_0\mid\bracks{X_i}_{-m}^{-1}=\bracks{\omega_i}_{-m}^{-1}}$ as shorthand for
\[
\PP\bracks{X_0=\omega_0\mid X_{-1}=\omega_{-1},\ldots,X_{-m}=\omega_{-m}}.
\]

\begin{defi}[Random-step Markov Process]\label{def:randomMarkov}
A stationary process $\left(\left(X_i\right)_{i \in \mathbb{Z}}, \mathbb{P}\right)$ on an alphabet $A$, with measurable sets $\mathcal{A}$, is called a \emph{random-step Markov process} (or \emph{random Markov process}) if{f} there exists an independent stationary process on the positive integers, $\left(L_i\right)_{i \in \mathbb{Z}}$, and a stationary coupling $\hat{\mathbb{P}}$ of $\left(X_i\right)_{i \in \mathbb{Z}}$ and $\left(L_i\right)_{i \in \mathbb{Z}}$ such that $L_0$ is independent of $\left\{X_i:i <0\right\}$, and so that for every $n \in \mathbb{Z}^+$, $\mathbf{\omega} \in A^{\mathbb{Z}}$ and measurable set $E_0 \subseteq A$,
\begin{multline}\label{eq:rm-def}
\hat{\PP}\bracks{X_0 \in E_0 \mid \bracks{X_i}_{i=-n}^{-1} = (\omega_i)_{i=-n}^{-1} \wedge L_0=n}\\
	=\hat{\PP}\bracks{X_0 \in E_0\mid \bracks{X_i}_{i<0} = \bracks{\omega_i}_{i<0} \wedge L_0=n}.
\end{multline}
The stationary coupling $\left(\left(X_i, L_i\right)_{i \in \mathbb{Z}}, \hat{\mathbb{P}}\right)$ is called a \emph{complete random-step Markov process} (or \emph{complete random Markov process}).
\end{defi}

For every $i \in \mathbb{Z}$, the variable $L_i$ is called the \emph{look back time (or distance) for $X_i$}, as one need know only $L_0$ and then look at the $L_0$-past of $(X_i)_{i \in \mathbb{Z}}$ to determine the law of $X_0$ exactly.  

Note that in Definition \ref{def:randomMarkov}, since conditional probabilities are only defined up to sets of measure $0$, it does not matter whether the condition \eqref{eq:rm-def} holds for all $\mathbf{\omega}$ or only almost all.  In the case that the alphabet $A$ is either finite or countably infinite, it suffices to verify condition \eqref{eq:rm-def} in the special case that the set $E_0$ consists of a singleton.

Previously, the types of stationary processes given by Definition \ref{def:randomMarkov} have been simply called `random Markov processes'.  In this paper, we shall continue to use this terminology, but specify `random-step Markov process' in the definition to clarify the source of the additional randomness, compared to a usual $n$-step Markov process.

A similar notion to that of a random Markov process is a `uniform martingale' also introduced by Kalikow~\cite{Kalikow90}, for which the martingale convergence theorem applies in a uniform way. That is, the distribution on $X_0$ given the past can be approximated uniformly and arbitrarily well by a distribution which is dependent only on the $m$-past, for some $m$ large enough.  

For a measure $\mu$, let $||\mu||_{\text{TV}}$ denote the total variation norm so that for probability measures $\mu$ and $\nu$, the total variation distance is given by $||\mu - \nu||_{\text{TV}} = \sup_E \left\{|\mu\left(E\right) - \nu\left(E\right)|\right\}$.  Recall also that if $\mu$ and $\nu$ are both measures on a countable set $A$, then $||\mu - \nu||_{\text{TV}} = \frac{1}{2} \sum_{a \in A}|\mu(a) - \nu(a)|$.

\begin{defi}\label{def:um}
A stationary distribution $\left(\left(X_i\right)_{i \in \mathbb{Z}}, \mathbb{P}\right)$ on alphabet $A$ is called a \emph{uniform martingale} if{f} for every $\varepsilon > 0$ there exists $n_{\varepsilon}$ so that for every $n \geq n_\varepsilon$ and every $\left(\omega_i\right)_{i<0} \in \prod_{i<0} A$,
\begin{equation}\label{eq:um-def}
\left\|\mathbb{P}\left(\cdot \mid \left(X_i\right)_{-\infty}^{-1} = \left(\omega_i\right)_{-\infty}^{-1}\right) - \mathbb{P}\left(\cdot \mid \left(X_i\right)_{-n}^{-1} = \left(\omega_i\right)_{-n}^{-1}\right)\right\|_{\text{TV}} <  \varepsilon.
\end{equation}

In the case that $A$ is countable, for every $k \geq 0$, define the \emph{$k$-th variation of $\mathbb{P}$} to be
\begin{align}
\var_k &= \var_k(\mathbb{P}) \notag\\ 
		&= \sup\bigg\{|\mathbb{P}\left(X_0 = a_0 \mid (X_i)_{i=-k}^{-1} = (\omega_i)_{i = -k}^{-1} \right) \notag\\
		& \qquad - \mathbb{P}\left(X_0 = a_0 \mid (X_i)_{i=-\infty}^{-1} = (\omega_i)_{i = -\infty}^{-1} \right) | \ :\ (\omega_i)_{i \leq -1} \in A^{\mathbb{Z}^-}\bigg\} \label{def:nvar}
\end{align}

Note that if $A$ is finite, this is equivalent to the condition that for every $\varepsilon>0$, there exists $n_{\varepsilon}$ so that for every $n \geq n_{\varepsilon}$, $a \in A$, and $\left(\omega_i\right)_{i<0} \in \prod_{i<0} A$,
\begin{multline}\label{eq:UM_finite}
\bigg|\mathbb{P}\left(X_0 = a \mid \left(X_i\right)_{i=-\infty}^{-1} = \left(\omega_i\right)_{i=-\infty}^{-1}\right)\\ - \mathbb{P}\left(X_0 = a \mid \left(X_i\right)_{i=-n}^{-1} =\left(\omega_i\right)_{i=-n}^{-1}\right)\bigg|
< \varepsilon.
\end{multline}
\end{defi}

The definition of the $k$-th variation in Equation \eqref{def:nvar} is chosen to agree with the definition of $k$-th variations of $g$-functions in Equation \eqref{def:nvar-g-fn} to come.

If $\left(\left(X_i\right)_{i \in \mathbb{Z}}, \mathbb{P}\right)$ is a random Markov process with look-back distances $(L_i)_{i \in \mathbb{Z}}$, then for every $n \geq 1$, and $(\omega_i)_{i < 0}$,
\begin{equation}\label{eq:tv-lookback}
 \| \mathbb{P}\left( \cdot \mid (X_i)_{-\infty}^{-1} = (\omega_i)_{-\infty}^{-1}\right) - \mathbb{P}\left(\cdot \mid (X_i)_{-n}^{-1} = (\omega_i)_{-n}^{-1}\right) \|_{TV} \leq  \mathbb{P}(L_0 > n),
\end{equation}
which tends to $0$ as $n$ tends to infinity, uniformly in the choice of $(\omega_i)_{i < 0}$.  In the original paper by Kalikow~\cite{Kalikow90}, the weaker condition in Equation \eqref{eq:UM_finite} was used as the definition of a uniform martingale.  In the case of processes on infinite alphabets, the stronger condition is necessary, as was noted in~\cite{Kalikow12}.

The reason for the name `uniform martingale' is that for any stationary process $((X_i)_{i \in \mathbb{Z}}, \mathbb{P})$ and every $a$, the sequence $\left\{\mathbb{P}\left(X_0 = a \mid \left(X_i\right)_{i=-m}^{-1} = \left(\omega_i\right)_{i=-m}^{-1}\right)\right\}_{m \geq 1}$ is a martingale and converges, pointwise, to 
\begin{equation}\label{eq:martingale}
\mathbb{P}\left(X_0 = a \mid \left(X_i\right)_{i=-\infty}^{-1} = \left(\omega\right)_{i=-\infty}^{-1}\right).
\end{equation}
A stationary process is a uniform martingale if this convergence is uniform in $\left(\omega_i\right)_{i<0}$.

Kalikow~\cite[Theorem 1.7]{Kalikow12} showed that, for any Lebesgue probability space, every aperiodic measure-preserving transformation is isomorphic to a random Markov process on a countable alphabet.  While the statement of Kalikow's Theorem 1.7~\cite{Kalikow12} concludes only that such transformations are isomorphic to uniform martingales, the proof, in fact, constructs a random Markov process.

In~\cite{Kalikow90}, Kalikow showed that a uniform martingale on a binary alphabet is also a random Markov process.  The purpose of this paper is to both strengthen this result for finite alphabets and to examine some extensions of this result to processes on both countable and uncountable alphabets.

In Section \ref{sec:ctble}, it is shown that uniform martingales on a countable alphabet that satisfy an additional condition, discussed below, are random Markov processes.

\begin{defi}[Dominating Measure]
\label{def:dommeasure}
Let $\bracks{\left(X_i\right)_{i \in \mathbb{Z}}, \prob}$ be a stationary process on an alphabet $A$ and $\mu$ a measure on $A$.  Then, $\mu$ is called a \emph{dominating measure} for $\bracks{\left(X_i\right)_{i \in \mathbb{Z}}, \prob}$ if{f} for every event $E$ in the $\sigma$-algebra generated by $\left(X_{i}\right)_{i<0}$ and every measurable set $A_0 \subseteq A$, $\prob\bracks{X_0 \in A_0 \mid E} \leq \mu\bracks{A_0}$.

If the measure $\mu$ is finite ($\mu\left(A\right)< \infty$), then $\left(\left(X_i\right)_{i \in \mathbb{Z}}, \prob\right)$ is said to have a \emph{finite dominating measure}.
\end{defi}

It was shown by Parry~\cite{Parry66, Parry69} and Rohlin~\cite{Rohlin65}  that every measure-preserving transformation is isomorphic to a stationary process on a countable alphabet for which the past determines the present.  These results were strengthened by Kalikow~\cite{Kalikow12}.  In the case of random Markov processes, the notion the past determining the present can sometimes be expressed precisely in terms of the look-back distances.   In particular, we shall consider the a particular class of random Markov processes, for which the look-back distance $L_0$ and the previous states uniquely determine the state $X_0$.

\begin{defi}
A random Markov process $\left(\left(X_i\right)_{i \in \mathbb{Z}}, \mathbb{P}\right)$ on a countable alphabet $A$ is called a \emph{deterministic random-step Markov process} if{f} there exists a representation as a complete random-step Markov process $\left(\left(X_i, L_i\right)_{i \in \mathbb{Z}}, \mathbb{P}'\right)$ so that for every sequence $\left(\omega_i\right)_{i \leq 0}$,
\begin{multline}\label{eq:def-det-rm}
\mathbb{P}'\left(X_0 = \omega_0 \mid \left(X_i\right)_{i <0} = \left(\omega_i\right)_{i<0} \wedge L_0 = n\right)\\ 
	=\mathbb{P}\left(X_0 = \omega_0 \mid \left(X_i\right)_{-n}^{-1} = \left(\omega_i\right)_{-n}^{-1} \wedge L_0 = n\right)\in \left\{0,1\right\}.
\end{multline}
\end{defi}

In this paper, only deterministic random Markov processes on countable alphabets are considered.  In particular, since the condition in equation \eqref{eq:def-det-rm} is trivially satisfied for many stationary processes on uncountable alphabets, a generalization of  the notion of a deterministic random Markov process to an uncountable alphabet would require a careful choice of definition, which is not addressed here.

In the study of random Markov processes on a countable alphabet, for any integer $n$ and sequence $\left(\omega_i\right)_{-n\leq i <0} \in A^n$, the values 
\begin{equation}\label{eq:table-vals}
\left\{\mathbb{P}\left(X_0 = \omega_0 \mid \left(X_i\right)_{i=-n}^{-1} = \left(\omega_i\right)_{i=-n}^{-1} \wedge L_0 = n\right) \mid \omega_0 \in A\right\}
\end{equation}
 are sometimes called the \emph{table values} for the event \[\left\{\left(X_i\right)_{i=-n}^{-1} = \left(\omega_i\right)_{i=-n}^{-1} \wedge L_0 = n\right\}\] (see~\cite{Kalikow90}).  A deterministic random Markov is then one with a representation as a complete random Markov process for which all table values are either $0$ or $1$.

While the property of being a deterministic random Markov process might appear to be quite strong, the following theorem shows that, in fact, all uniform martingales on finite alphabets can be expressed in this form.

\begin{thm}\label{thm:countable}
Every uniform martingale, $\left(\left(X_i\right)_{i \in \mathbb{Z}}, \mathbb{P}\right)$, on a countable alphabet with a finite dominating measure, $\mu$, is a deterministic random Markov process.

In the case that the alphabet $A$ is finite, $\left(\left(X_i\right)_{i \in \mathbb{Z}}, \mathbb{P}\right)$ is a deterministic random Markov process with finite expected look-back distance if{f} the $n$-th variations are summable:
\[
\sum_{n \geq 1} \var_n (\mathbb{P}) < \infty.
\] 
\end{thm}

As a corollary, note that if $A$ is a finite set, every stationary process on $A^{\mathbb{Z}}$ has a finite dominating measure: for example, for each $a \in A$, set $\mu\left(a\right) = 1$ so that $\mu\left(A\right) = |A| < \infty$.  Thus, Theorem \ref{thm:countable} simplifies for processes on a finite alphabet.

\begin{cor}[Finite Alphabets]
\label{thm:ergfinlet}
Let $A$ be a finite set and $\left(\left(X_i\right)_{i \in \mathbb{Z}}, \mathbb{P}\right)$ be a uniform martingale.  Then $\left(\left(X_i\right)_{i \in \mathbb{Z}}, \mathbb{P}\right)$ is a deterministic random Markov process and furthermore has a representation as a complete random Markov process with finite expect look-back distance if{f} $\sum_n \var_n (\mathbb{P}) < \infty$.
\end{cor}

Theorem \ref{thm:countable} is best-possible in the sense that there are uniform martingales on countable alphabets without a dominating measure that are not random Markov processes.  Such an example is given in Section \ref{sec:ctble}.

For uncountable alphabets, there are examples of uniform martingales on an uncountable alphabet with a finite dominating measure that are not random Markov processes.  Such an example is given in Section \ref{sec:unctble}, where a stronger condition (Berbee's ratio condition, Definition \ref{cond:four} below) is considered that implies that a process is a uniform martingale and also a random Markov process. 

Berbee~\cite{Berbee87} considered Markov representations of stationary processes, based on the properties of their associated $g$-functions.  The condition on $g$-functions that was considered by Berbee in \cite{Berbee87} was different than that for uniform martingales (Definition \ref{def:um}).  Given a stationary process $((X_i)_{i \in \mathbb{Z}}, \mathbb{P})$ and $n \geq 1$, define
\begin{multline}\label{eq:Berbee-cond}
r_n = r_n(\mathbb{P}) = \log\bigg(\sup\bigg\{\frac{\mathbb{P}(X_0 = x_0 \mid X_{-1} =  x_{-1}, \ldots)}{\mathbb{P}(X_0 = y_0 \mid X_{-1} = y_{-1}, \ldots)} \bigg\rvert \\ x_0 = y_0, x_{-1} = y_{-1}, \ldots, x_{-n} = y_{-n} \bigg\} \bigg).
\end{multline}
Reframing the results in the language of random Markov processes, Berbee showed that if $\left(\left(X_i\right)_{i \in \mathbb{Z}}, \mathbb{P}\right)$ is a uniform martingale on a finite alphabet with $\sum_n r_n < \infty$, then $\left(\left(X_i\right)_{i \in \mathbb{Z}}, \mathbb{P}\right)$ is a random Markov process.   Theorem \ref{thm:countable} provides a stronger result when applied to a process on a finite alphabet and Theorem \ref{thm:ratio} below shows that a closely related condition is sufficient for any process, even on an uncountable alphabet, to be a random Markov process.

The following definition (Definition \ref{cond:four} below) is equivalent in the case of a countable alphabet to the underlying stationary process having $\{r_n\}_{n \geq 1}$ as in Equation \ref{eq:Berbee-cond} satisfying $\lim_{n \to \infty} r_n = 0$.   A stationary process satisfying the condition in Definition \ref{cond:four} is some times also called \emph{log-continuous}, due to the formulation from Equation \eqref{eq:Berbee-cond}. 

\begin{defi}[Berbee's ratio condition]
\label{cond:four}
A stationary process on an alphabet $A$ is said to satisfy \emph{Berbee's ratio condition} (or simply the \emph{ratio condition}) if{f} for every $\varepsilon >0$ there exists $n = n(\varepsilon)$ so that for every $(\omega_i)_{i<0} \in A^{\mathbb{Z^-}}$ and every measurable set $E_0 \subseteq A$,
\begin{equation}\label{eq:ratiodef}
\left| \frac{\PP\bracks{X_0 \in E_0 \mid\bracks{X_i}_{i\in\ZZ_{-}}=\bracks{\omega_i}_{i\in\ZZ_{-}}}}{\PP\bracks{X_0 \in E_0\mid \bracks{X_{i}}_{-m}^{-1}=\bracks{\omega_{i}}_{-m}^{-1}}} - 1\right| < \varepsilon.
\end{equation}

Note that in the expression in equation \eqref{eq:ratiodef}, a $0$ only appears in the denominator when there is also a $0$ in the numerator.  The convention adopted for this possible scenario is that $\frac{0}{0} = 1$.
\end{defi}

As in the definition of a random Markov process, Definition \ref{def:randomMarkov}, it suffices that the condition in \eqref{eq:ratiodef} hold for merely almost all $\mathbf{\omega}$.

Examples are given in Section \ref{sec:unctble} to show that not every uniform martingale satisfies the ratio condition and not every stationary process that satisfies the ratio condition has a dominating measure.  However, even in the case where the alphabet is uncountable, every stationary process that satisfies Berbee's ratio condition is a random Markov process; this is stated formally in the following theorem. 

\begin{thm}\label{thm:ratio}
Let $A$ be any set and $\Omega = A^{\mathbb{Z}}$.  Let $\mathbb{P}$ be a stationary probability measure on $\Omega$ so that $\left(\Omega, \mathbb{P}\right)$ satisfies the ratio condition, then $\left(\Omega, \mathbb{P}\right)$ is a random Markov process.
\end{thm}

Finally, some open questions and conjectures are given in Section \ref{sec:open}.

\subsection{Background}
\label{sec:gfunctions}

Random Markov processes and uniform martingales, both introduced by Kalikow~\cite{Kalikow90}, are related to the study of `$g$-functions'.  Roughly, a $g$~-~function determines the probability distribution on the alphabet given a possibly infinite word.  These are also known as transition probabilities.  This section gives a short background on both uniform martingales and (continuous) $g$-functions, including, for completeness, some sketches of known constructions for measure for certain $g$-functions which are used throughout. It should be noted that this is merely an overview; we do not attempt to survey this area in its entirety. For further information, we point the reader towards the many references given within this section.

The notion of a probability distribution on the present state of a process depending on infinitely many past states was introduced by Onicescu and Mihoc \cite{OM35} under the name `cha\^{i}nes \`{a} liaisons compl\`{e}tes' (now called `chains with complete connections') and subsequently developed by Doeblin and Fortet~\cite{DF37}.  Connections between a number of closely related definitions are given in lecture notes by Maillard \cite{Maillard07}.  The terminology and notation used in this paper are as follows.
 
\begin{defi}
For a countable set $A$, a function $f : A \times \prod_{i<0} A \to [0,1]$ is called a \emph{$g$-function} if{f} for every $\left(\omega_i\right)_{i<0} \in A^{\mathbb{Z}^-}$,
\begin{equation}\label{eq:g-function}
\sum_{a \in A} f\left(a, \omega_{-1}, \omega_{-2}, \ldots\right) = 1.
\end{equation}
\end{defi}

Every stationary process $\left(\left(X_i\right)_{i \in \mathbb{Z}}, \mathbb{P}\right)$ on a countable alphabet $A$ corresponds to a $g$-function defined by
\begin{equation}\label{eq:g-measure}
f\left(a, \omega_{-1}, \omega_{-2}, \ldots\right) = \mathbb{P}\left(X_0 = a \mid X_{-1} = \omega_{-1}, X_{-2} = \omega_{-2}, \ldots\right).
\end{equation}
For a particular $g$-function $f$, any probability measure that satisfies equation \eqref{eq:g-measure} is called a $g$-measure for $f$.  For every $k$, define the \emph{$k$-th variation of $f$} by 
\begin{multline}\label{def:nvar-g-fn}
\var_k\left(f\right) = \sup\left\{\left|f\left(x_0, x_{-1}, x_{-2}, \ldots\right) - f\left(y_0, y_{-1}, y_{-2}, \ldots\right)\right|\right.\\
\left.\lvert x_0 =y_0, x_{-1} = y_{-1}, \ldots, x_{-k} = y_{-k}\right\}.
\end{multline}
If $f$ satisfies equation \eqref{eq:g-measure} for a stationary process $\left(\left(X_i\right)_{i \in \mathbb{Z}}, \mathbb{P}\right)$, write $\var_k (\mathbb{P})$ for $\var_k\left(f\right)$.  When the $g$-function in question is clear from context, we shall sometimes write $\var(k)$.

The notion of a $g$-function is reserved here for processes on a countable alphabet.  

Keane~\cite{Keane72} looked at those $g$-functions that are continuous as linear forms on the space of probability measures for a compact metric space.  For stationary distributions on a finite alphabet, a uniform martingale corresponds precisely to a continuous $g$-function.  In the same paper, Keane proved that if $A$ is finite and $f$ is a continuous $g$-function, there exists at least one $g$-measure for $f$ and gave certain sufficient conditions for uniqueness of these measures.  Since then, there have been a number of results related to the existence and uniqueness of $g$-measures for a particular $g$-function~\cite{Berbee87, BM93, Hulse06, JOP07, Keane72, MU01, Sarig99}.  In the remainder of this section, we review some results related to $g$-functions.

  Given an alphabet $A$ and a stochastic process $\left(X_i\right)_{i \in \mathbb{Z}} \in A^{\mathbb{Z}}$ with probability measure $\mathbb{P}$, the function $f: \prod_{i \leq 0} A~\to~[0,1]$ given by
\begin{equation}\label{eq:marginal_fn}
 f\left(a_0, a_{-1}, a_{-2}, \ldots\right) = \mathbb{P}\left(X_0 = a_0 \mid X_{-1}=a_{-1}, X_{-2} = a_{-2}, \ldots\right)
\end{equation}
is uniquely determined up to sets of $\mathbb{P}$-measure $0$.  By definition, for any $\mathbf{a} \in \prod_{i <0} A$, the function $f\left(\cdot\ \mathbf{a}\right): A \to [0,1]$ determines a probability measure on $A$.  If $\left(X_i\right)_{i \in \mathbb{Z}}$ is stationary then also, for any $i \in \mathbb{Z}$,
\begin{equation}\label{eq:stationary}
 \mathbb{P}\left(X_i = a_0 \mid X_{i-1}=a_{-1}, X_{i-2} = a_{-2}, \ldots\right) = f\left(a_0, a_{-1}, a_{-2}, \ldots\right).
\end{equation}

Further, when $|A| < \infty$, the set $\prod_{i \leq 0} A$ can be endowed with a metric and a compact topology.

Doeblin and Fortet~\cite{DF37} examined these types of functions 
\[
\mathbf{a} \mapsto \PP\left(X_0 = a_0 \mid X_{-1} = a_{-1}, X_{-2} = a_{-2}, \ldots\right),
\]
 which they called \emph{chains with infinite connections} and gave some results on properties of some limits of these functions, under certain conditions.

A function $f$, as in equations \eqref{eq:marginal_fn} and \eqref{eq:stationary}, need not necessarily arise from a fixed probability measure.  Keane~\cite{Keane72} looked at the question of determining the collection of probability measures $\mu$ on $A^{\mathbb{Z}}$ for which 
\[
\mu\left(X_0 =a_0 \mid X_{-1} = a_{-1}, \ldots\right) = f\left(a_0, a_{-1}, a_{-2}, \ldots\right),
\]
 and in particular, under which conditions there is exactly one such measure.

The use of the letter $g$ in the terms `$g$-function' and `$g$-measure' seems to originate in a paper of Keane~\cite{Keane72}. There, the letter $g$ was used for the functions on $\prod_{i<0} A$, and a `$g$-measure' was one that was associated with precisely the function $g$.  Later, the term `$g$-function' came into use, to describe this class of functions.

As previously noted, if $f$ is a continuous $g$-function and $\mu$ is a $g$-measure for $f$, then by definition, $\mu$ is a uniform martingale.

Markov processes are a particular example of continuous $g$-functions.  Given a transition matrix $P = \left(p_{a, b}\right)_{a, b, \in A}$, for a finite state space $A$, the function $f$ given, for $a, b \in A$ and $\mathbf{x} \in A^-$, by $f\left(b, a, \mathbf{x}\right) = p_{a, b}$ is a continuous $g$-function for $A^{\mathbb{Z}^-}$.  

Note that, even in the case where $A = \left\{0,1\right\}$, if a $g$-function $f$ is not continuous, there need not be any $g$-measures for $f$, as the following example shows.

\begin{example}
 Let $A=\left\{0,1\right\}$ and define $f\left(a, \mathbf{x}\right)$ as follows:  If $\mathbf{x}$ contains an infinite string of consecutive $0$s, set $f\left(1, \mathbf{x}\right) = 1$ and $f\left(0, \mathbf{x}\right) = 0$.  If $\mathbf{x}$ does not contain an infinite string of $0$s, set $f\left(1, \mathbf{x}\right)=0$ and $f\left(0, \mathbf{x}\right)=1$.

The function $f$ is not continuous and one can check that there is no stationary $g$-measure for this $g$-function.\qed
\end{example}
 
Using fixed-point theorems, Keane~\cite{Keane72}, showed that every continuous $g$-function on a finite alphabet has at least one $g$-measure.  Such probability measures can also be constructed directly by taking limits of `nearly-stationary' measures on finite words constructed from a $g$-function.  Such a method for constructing stationary processes can be found, for example, in the book of Kalikow and McCutcheon~\cite[Section 2.11]{KMcC10}.

\begin{prop}[Keane~\cite{Keane72}]\label{Prop:existence}
 Let $(A, d)$ be a compact metric space and let $f:A^{\mathbb{Z}^-} \to [0,1]$ be a $g$-function that is continuous in the product topology on $A^{\mathbb{Z}^-}$.  There is at least one stationary $g$-measure for $f$.
\end{prop}

If the alphabet $A$ is finite, then the set $A$ with the trivial metric is compact and $g$-measures for continuous $g$-functions correspond exactly to uniform martingales.

In general, this type of result need not be possible in the case of infinite alphabets: the same technique need not yield a probability measure, nor even a non-zero measure.  For example, one can define a Markov process on $\mathbb{Z}$ in terms of a $g$-function that has no stationary measure.

In this paper, a number of examples are described in terms of their $g$-functions and these types of existence results are needed to verify that corresponding $g$-measures exist.  As some of the properties of these $g$-measures are used in the proofs, in practice, most examples of stationary processes in this paper are constructed explicitly by defining measures on all cylinder sets.

One of the methods used repeatedly in this paper to construct stationary processes is one that the third author learned from Ornstein in about 1978.  This construction, that we shall describe shortly, was used by Alexander and Kalikow \cite{AK92} to study randomly generated stationary processes.  This method recursively constructs stationary measures on finite words in the alphabet as follows.  Given an alphabet $A$, for each $n \geq 1$, a stationary probability measure $\mu_n$ is defined on $A^n$ so that the sequence of measures $\{\mu_n\}_{n \geq 1}$ is consistent.  The recursive construction begins by choosing a probability measure $\mu_1$ on $A$.  For every $n \geq 1$, given a stationary probability measure, $\mu_n$, on $A^n$, a consistent stationary measure on $A^{n+1}$ is defined in the following manner.  For every word of length $n-1$, $a_1 a_2 \ldots a_{n-1} \in A^{n-1}$, the measure $\mu_n$ gives conditional measures on all words of the form $x_0 a_1 a_2 \ldots a_{n-1}$ and all words of the form $a_1 a_2 \ldots a_{n-1} x_n$.  
Coupling these two conditional measures in any way, and even differently for different words $a_1 a_2 \ldots a_{n-1}$ yields a measure $\mu_{n+1}$ on $A^{n+1}$ with the property that if $\mathbf{a} \in A^n$, then
\begin{equation}\label{eq:measure-constr}
\mu_{n+1}(\cdot \mathbf{a}) = \mu_{n+1}(\mathbf{a} \cdot) = \mu_n(\mathbf{a}).
\end{equation}
Thus, the measure $\mu_{n+1}$ is stationary since $\mu_n$ is stationary and these measures are consistent.  Note that in the step $n =1$ of the recursion, the unique word of length $n-1$ is the empty word.

This method is as general as possible since any stationary process can be constructed in this way.  Furthermore, in many cases, properties of the resulting measure can be incorporated into the recursion and proved along the way.  In the case that $A$ is either finite or countable, the construction of the measures on finite words is straightforward and can be given by defining the probability of each word individually.  For uncountable alphabets, constructing the measures on finite words may require some care, but in either case, once the collection of stationary probability measures on all finite words is given, these extend to a stationary probability measure on $A^{\mathbb{Z}}$ with all probabilities of cylinder sets given by the appropriate measure on finite words.

In much of the literature, the focus has been on uniqueness of $g$-measures, rather than simply existence.  In his paper introducing the notion of $g$-functions, Keane~\cite{Keane72} showed that if $A$ is finite and a function $f$ is both continuous and satisfies certain differentiability conditions, then there is a unique measure for $f$ that is also strongly mixing.  Further conditions for uniqueness were given by Berbee~\cite{Berbee87} and Hulse~\cite{Hulse06} for $g$-function on finite alphabets.  Results for existence and uniqueness on countable alphabets were considered by Johansson, \"{O}berg, and Pollicott~\cite{JOP07}, Sarig~\cite{Sarig99}, and Mauldin and Urba\'{n}ski~\cite{MU01}.  Examples of $g$-functions with more than one $g$-measure were given by Bramson and Kalikow~\cite{BM93}, by Hulse~\cite{Hulse06}, and by Berger, Hoffman and Sidoravicius~\cite{BHS08}.

A random Markov process on a countable alphabet can be expressed in terms of $g$-functions.  Let $\left(\left(X_n, L_n\right)_{n \in Z}, \mathbb{P}\right)$ be a random Markov process on alphabet $A$.  For each $n \geq 1$, and $a_0, a_{-1}, \ldots, a_{-n} \in A$, define
\begin{multline*}
f_n\left(a_0, a_{-1}, \ldots, a_{-n}\right) =\\ \mathbb{P}\left(X_0 = a_0 \mid X_{-1} = a_{-1}, \ldots, X_{-n} = a_{-n}\ \wedge L_0 = n\right ).                                                                                                                                                                                                                                                                                                                                                                                                                                                                                                                                                                                                                                                                                                                                                                                                                               \end{multline*}
Such a function $f_n$ is naturally extended to a continuous $g$-function on $A^{\mathbb{Z}^-}$ and by the definition of a random Markov process (Definition \ref{def:randomMarkov}),
\begin{align*}
 \mathbb{P}&\left(X_0=a_0 \mid X_{-1}=a_{-1}, X_{-2} = a_{-2}, \ldots\right) \\ 
 &=\sum_{n=1}^{\infty} \mathbb{P}\left(L_0 = n\right)\cdot f_n\left(a_0, a_{-1}, \ldots, a_{-n}\right)\\
  &= f\left(a_0, a_{-1}, a_{-2}, \ldots\right).
\end{align*}
In this way, a random Markov processes is a random mix of $n$-step Markov processes.  To see that this $g$-function $f$ is continuous, for every $\varepsilon >0$, let $N$ be such that for $M \geq N$, $\sum_{n \geq M} \mathbb{P}\left(L_0 =n\right) < \varepsilon$.  Then,
\[
 \left|f\left(a_0, a_{-1}, \ldots\right) - \sum_{n=1}^{M-1}\mathbb{P}\left(L_0=n\right)\cdot f_n\left(a_0, a_{-1}, \ldots, a_{-n}\right)\right|< \varepsilon.
\]
In particular, given a random Markov process, any other stationary process with the same $g$-function will also be a random Markov process.

To see that the notion of random Markov processes is genuinely different from $n$-step Markov processes, consider the following example.

\begin{example}
Define a $g$-function on the alphabet $\{0,1\} \times \mathbb{Z}$ as follows.  For each $a_0, a_{-1}, a_{-2}, \ldots \in \{0,1\}$ and $\ell_0, \ell_{-1}, \ell_{-2}, \ldots \in \mathbb{Z}^+$, define
\begin{equation}\label{eq:rm-notMarkov}
f((a_0, \ell_0), (a_{-1}, \ell_{-1}), (a_{-2}, \ell_{-2}), \ldots) = 
\begin{cases}
	\frac{1}{4} \cdot \frac{1}{2^{\ell_0}	}	&\text{if } a_0 = a_{-\ell_0}\\
	\frac{3}{4} \cdot \frac{1}{2^{\ell_0}	}	&\text{if } a_0 = 1 - a_{-\ell_0}.
\end{cases}
\end{equation}
One can show that there is a complete random Markov process $\left(\left(X_i, L_i\right)_{i \in \mathbb{Z}}, \mathbb{P}\right)$ that is a $g$-measure for the $g$-function in equation \eqref{eq:rm-notMarkov} so that for every $k \geq 1$, $\mathbb{P}(L_0 = k) = \frac{1}{2^k}$ and so that given $L_0 = k$ and $X_{-k}$, $X_0$ is equal to $X_{-k}$ with probability $1/4$ and different with probability $3/4$.  It is left as an exercise to the interested reader to show that this can be done with $\mathbb{P}\left(X_0 = 0\right) = \mathbb{P}\left(X_0=1\right)=\frac{1}{2}$.  Such a stationary process is a random Markov process, but not an $n$-step Markov chain for any $n$.
\qed \end{example}

Further consideration of the properties of random Markov processes were given by Kalikow, Katznelson and Weiss~\cite{KKW92}, who showed that zero entropy systems can be extended to a random Markov process.  Later, Rahe~\cite{Rahe93, Rahe94} gave extensions of some results for $n$-step Markov chains to random Markov chains for which the expected look-back distance is finite.

\section{Countable alphabets}
\label{sec:ctble}
\subsection{Proof of Theorem \ref{thm:countable}}

In this section the proof of Theorem \ref{thm:countable} is given showing that if a uniform martingale on a countable alphabet has a finite dominating measure, then it is also a deterministic random Markov process.  Recall that for the $n$-th variation (see equation \eqref{def:nvar}) of a $g$-function associated with a particular stationary process, we shall write $\var_n \mathbb{P}$ or $\var(n)$ interchangeably.

We now recall Theorem \ref{thm:countable} and give its proof in full.

\begin{thm1}
Let $\left(\left(X_i\right)_{i \in \mathbb{Z}}, \mathbb{P}\right)$ be a uniform martingale on a countable alphabet $A$ with a dominating measure $\mu$.  Then 
\begin{enumerate}[(a)]
	\item \label{part:count-determ} $\left(\left(X_i\right)_{i \in \mathbb{Z}}, \mathbb{P}\right)$ is a deterministic random Markov process, and further
	\item \label{part:finite-look-back} if the alphabet $A$ is finite, then the process is a deterministic random Markov process with finite expected look-back distance if{f} $\sum_{n \geq 1} \var_n \mathbb{P} < \infty$. 
\end{enumerate}
\end{thm1}

\begin{proof}
Let $\left(\left(X_i\right)_{i \in \mathbb{Z}}, \mathbb{P}\right)$ be a uniform martingale on a countable alphabet $A$.  Without loss of generality, assume that $A = \mathbb{Z}^+$.  First, to prove part \eqref{part:count-determ}, the random variables $\left(L_i\right)_{i \in \mathbb{Z}}$ are constructed recursively together with a stationary coupling that is a deterministic random Markov process.

Indeed, the table values, as in equation \eqref{eq:table-vals}, for the random Markov process are constructed in terms of a sequence of functions $(T_k)_{k \geq 0}$ together with a sequence of look-back distances $\{n_k\}_{k \geq 0}$ and a sequence of look-back probabilities $\{p_k\}_{k \geq 0}$, constructed recursively with the following properties:
\begin{enumerate}[(i)]
	\item \label{rec:unique} For every $k \geq 0$, $T_k : A^{n_k+1} \to [0,1]$ has the property that for every $\mathbf{\omega} \in A^{n_k}$, there is $a(\mathbf{\omega}) \in A$ with $T_k(a(\mathbf{\omega}), \mathbf{\omega}) = p_k$ and if $b \neq a(\mathbf{\omega})$, then $T_k(b, \mathbf{\omega}) = 0$.
	\item \label{rec:leq} For every $(\omega_i)_{i \leq -1} \in A^{\mathbb{Z}^-}$ and $b \in A$,
	\begin{equation}\label{eq:tableValues}
\sum_{i \leq k} T_i(b; (\omega_i)_{-n_i}^{-1}) \leq \mathbb{P}\left(X_0 = b \mid (X_i)_{i=-\infty}^{-1} = (\omega_i)_{i=-\infty}^{-1}\right).
\end{equation}
	\item \label{rec:converge} For every $k \geq 1$, there exists $(\omega_{k, i})_{i =- n_k}^{-1} \in A^{n_k}$ so that for every $b \in A$,
	\begin{equation}\label{eq:table-conv}
	\mathbb{P}\left(X_0 = b \mid (X_i)_{i=-n_{k}}^{-1} = (\omega_{k, i})_{i = -n_k}^{-1}\right) - \sum_{i = 1}^{k-1} T_i(b, (\omega_{k,i})_{i=-n_i}^{-1}) \leq p_k + \var(n_k) + \frac{1}{k}.
	\end{equation}
\end{enumerate}

Note that since the process is a uniform martingale, the sequence of $n$-th variations $(\var(n))_{n \geq 1}$ is non-increasing and satisfies $\lim_{n \to \infty} \var(n) = 0$.  It is possible that for some $N \geq 1$, $\var(N) = 0$ in which case the process is an $N$-step Markov process.  The existence of such an $N$ has no effect on the following construction.

To begin the recursive construction for $k = 0$, set $p_0 = n_0 = 0$ and let $T_0 \equiv 0$ (the $0$-function).  Note that the conditions in parts \eqref{rec:unique} and \eqref{rec:leq} are trivially satisfied and the condition in part \eqref{rec:converge} does not apply to $k = 0$.

For the recursion step, fix $k \geq 1$ and suppose that $T_0, T_1, \ldots, T_{k-1}$, $p_0, p_1, \ldots, p_{k-1}$ and $n_0, n_1, \ldots, n_{k-1}$ have been defined and satisfy conditions \eqref{rec:unique}, \eqref{rec:leq}, and \eqref{rec:converge} above.

For ease of notation, for any $n \geq n_{k-1}$ and $(\omega_i)_{i = -n}^{-1} \in A^n$ or $(\omega_i)_{i=-\infty}^{-1} \in A^{\mathbb{Z}^-}$ and $b \in A$, define
\begin{equation}\label{eq:left-over-measure}
P_{k-1}(b; (\omega_i)_{i \leq -1}) = \mathbb{P}\left(X_0 = b \mid (X_i)_{i\leq -1} = (\omega_i)_{i\leq -1}\right) - \sum_{i \leq k-1} T_i(b; (\omega_i)_{-n_i}^{-1}).
\end{equation}

By the condition \eqref{rec:leq}, for any $b$ and $(\omega_i)_{i \leq -1}$, then $P_{k-1}(b, (\omega_i)_{i \leq -1}) \geq 0$.  For any $(\omega_i)_{i \leq -1} \in A^{\mathbb{Z}^-}$ or $(\omega_i)_{i=-n}^{-1} \in A^n$, the function defined by $P_{k-1}(\cdot, (\omega_i)_{i \leq -1})$ is a positive measure on $A$ that is dominated by the measure $\mu$.  Also, for any $n \geq n_{k-1}$ and $\mathbf{\omega}, \mathbf{\omega}' \in A^{\mathbb{Z}^-}$ with $\omega_{-1} = \omega_{-1}', \ldots, \omega_{-n} = \omega_{-n}'$, then
\begin{equation}\label{eq:left-over-variation}
|P_{k-1}(b; \mathbf{\omega}) - P_{k-1}(b; \mathbf{\omega}')| \leq \var(n).
\end{equation}

Set $\gamma = 1 - \sum_{j = 0}^{k-1}p_j$ and let $M > 0$ be sufficiently large so that the dominating measure satisfies $\sum_{k > M} \mu\left(k\right) < \frac{\gamma}{10}$.  Then, for any $\left(\omega_i\right)_{i < 0} \in A^{\geq n_{k-1}}$ and any $n \geq n_{k-1}$,
\[\sum_{a =1}^{M}P_{k-1}(a; \left(\omega_i\right)_{-n}^{-1}) \geq \frac{9\gamma}{10}.\]
In particular, for each $n \geq n_{k-1}$ and $\left(\omega_i\right)_{-n}^{-1}$, there is an $a \in A$, depending on $(\omega_i)_{i=-n}^{-1}$ so that
\[
P_{k-1}\left(a;  \left(\omega_i\right)_{-n}^{-1}\right) \geq \frac{9\gamma}{10M}.
\]
Choose $n_k > n_{k-1}$ to be large enough so that $2\var(n_k) \leq \frac{9\gamma}{10M}$. That is, for every $\left(\omega_i\right)_{-n}^{-1}$, there is an $a \in A$ with
\begin{equation}\label{eq:likely-element}
P_{k-1}\left(a;  \left(\omega_i\right)_{-n}^{-1}\right) > 2 \var(n_k).
\end{equation}

For every $\mathbf{\omega} = \left(\omega_i\right)_{-n_k}^{-1} \in A^{n_k}$, define
\[
s_k\left(\mathbf{\omega}\right) = \max_{a \in A} \left\{P_{k-1}(a; \left(\omega_i\right)_{-n_k}^{-1})\right\}
\]
and let $a\left(\mathbf{\omega}\right) \in A$ achieve this maximum.  Set 
\begin{equation}\label{eq:rk}
r_k = \inf \left\{ s_k\left(\mathbf{\omega}\right) \mid \mathbf{\omega} \in A^{n_k}\right\}
\end{equation}
and define $p_k = r_k - \var\left(n_k\right)$.  By the choice of $n_k$ and inequality \eqref{eq:likely-element}, $p_k \geq 0$.  For any $a \in A$ and $\mathbf{\omega} \in A^{n_k}$, define
\[
T_k\left(a; \omega_{-1}, \ldots, \omega_{-n_k}\right) = 
	\begin{cases}
		p_k	&\text{if $a = a\left(\mathbf{\omega}\right)$}\\
		0	&\text{otherwise}.
	\end{cases}
\]                                                                                                                                                                                                                                                                                                                            
For any $\left(\omega_i\right)_{-\infty}^{-1} \in A^{\mathbb{Z}^-}$ and $a \in A$, by the choice of $n_k$, inequality \eqref{eq:likely-element}, and the definition of $p_k$, 
\[
T_k\left(a; \omega_{-1}, \ldots, \omega_{-n_k}\right) \leq P_{k-1}\left(a; \left(\omega_i\right)_{i=-\infty}^{-1}\right).
\] 
Thus, by the definition of $P_{k-1}$ (equation \eqref{eq:left-over-measure}), condition \eqref{rec:leq} is satisfied for this value of $k$.  Further, by the definition of $r_k$ in equation \eqref{eq:rk} and the definition of $p_k = r_k - \var(n_k)$, there exists some $\mathbf{\omega}_k$ so that for every $b \in A$,
\[
P_{k-1}(b; \mathbf{\omega}_k) \leq s(\mathbf{\omega}_k) \leq r_k + \frac{1}{k} = p_k + \var(n_k) + \frac{1}{k}.
\]
Thus, condition \eqref{rec:converge} is satisfied for this value of $k$ also.

This completes the recursive construction.  Given $(T_i)_{i \geq 1}$, a representation of the random Markov process is defined by setting, for every $k \geq 1$, $\hat{\mathbb{P}}\left(L_0 = n_k\right) = p_k$ and for every $a \in A$ and $\mathbf{\omega} \in A^{n_k}$,
\[
\hat{\mathbb{P}}\left(X_0 = a  \mid \left(X_i\right)_{i=-n_k}^{-1} = \left(\omega_i\right)_{i=-n_k}^{-1}\wedge L_0 = n_k\right) = \frac{1}{p_k}T_{k}\left(a; \omega_{-1}, \ldots, \omega_{-n_{k}}\right).
\]
Note that the functions given by $\frac{1}{p_k} T_k$ are the table values for the random Markov process being defined.

It remains to show that $\sum_{k \geq 1} p_k = 1$ and that 
\begin{multline}\label{eq:same-measure}
\sum_k \hat{\mathbb{P}}\left(X_0 = a \wedge L_0 = n_k \mid \left(X_i\right)_{-n_k}^{-1} = \left(\omega_i\right)_{-n_k}^{-1}\right)\\ = \mathbb{P}\left(X_0 = a \mid \left(X_i\right)_{-\infty}^{-1} = \left(\omega_i\right)_{-\infty}^{-1}\right).
\end{multline}

Note that, by construction $\sum_{k \geq 1} p_k \leq 1$ which implies that $\lim_{k \to \infty} p_k = 0$.  For every $k \geq 1$, let $\mathbf{\omega}_k \in A^{n_k}$ be defined so that $s_k\left(\mathbf{\omega}_k\right) \leq  r_k + \frac{1}{k} = p_k+\var\left(n_k\right) + \frac{1}{k}$.  Then, for every $a \in A$,
\[
P_k\left(a; \mathbf{\omega}_k\right) \leq s_k \leq p_k + \var\left(n_k\right) + \frac{1}{k}.
\]
Since $\lim_{k \to \infty} \left(p_k + \var\left(n_k\right) + \frac{1}{k}\right) = 0$, the sequence $\left(P_k\left(\cdot; \mathbf{\omega}_k\right)\right)_{k \geq 1}$ consists of positive measures on the countable set $A$, each dominated by $\mu$, converging to $0$ pointwise.  By the Lebesgue dominated convergence theorem, 
\[
\lim_{k \to \infty} 1 - \sum_{i=1}^k p_i = \lim_{k \to \infty}\sum_{a \in A} P_k\left(a; \mathbf{\omega}_k\right) = 0.
\] 
Thus, $\sum_{k \geq 0} p_k = 1$ and hence $\sum_{k \geq 1} \hat{\mathbb{P}}\left(X_0  = \cdot \wedge L_0 = n_k \mid \left(X_i\right)_{-n_k}^{-1} = \left(\omega_i\right)_{-n_k}^{-1}\right)$ is a probability measure on $A$ and so equation \ref{eq:same-measure} holds.

Consider now the expected value of the look-back distance for part \eqref{part:finite-look-back} of the theorem.  First, let $\left(\left(X_i, \hat{L}_i\right)_{i \in \mathbb{Z}}, \mathbb{P}\right)$ be a random Markov process with $\mathbb{E}\left(\hat{L}_0\right) < \infty$.  For every $n$,
$\var\left(n\right) \leq 2\cdot\mathbb{P}\left(\hat{L_0} > n\right)$ and so
\[
\sum_{n \geq 1} \var\left(n\right) \leq 2 \sum_{n \geq 1} \mathbb{P}\left(\hat{L}_0 > n\right) \leq 2\cdot\mathbb{E}\left(\hat{L}_0\right) < \infty.
\]

In the case that $A$ is finite, let $M \in \mathbb{Z}^+$ be such that $|A| = M < \infty$ and suppose that the process satisfies $\sum_{n \geq 1} \var(n) < \infty$.  The recursive construction given can be repeated as above, with the following changes.  For every $k \geq 1$, choose $n_k> n_{k-1}$ to be the smallest integer $n$ with
\begin{equation}\label{eq:careful_nk}
\min \left\{
	\inf_{\mathbf{\omega} \in A^n} 
		\left\{\max_{a \in A} P_{k-1}(a; \mathbf{\omega})\right\},
		 \left(1- \frac{1}{M^2}\right)r_{k-1}
	\right\} \geq 2\var(n).
\end{equation}
Such an $n$ is well-defined since the left-hand side of inequality \eqref{eq:careful_nk} is increasing in $n$ and bounded away from $0$ while the right-hand side is decreasing to $0$.  Define
\[
r_k = \min \left\{
	\inf_{\mathbf{\omega} \in A^n} 
		\left\{\max_{a \in A} P_{k-1}(a; \mathbf{\omega})\right\},
		 \left(1- \frac{1}{M^2}\right)r_{k-1}
	\right\}
\]
and set $p_k = r_k - \var(n_k)$.  By inequality \eqref{eq:careful_nk}, $p_k > 0$ and also $r_k \leq r_{k-1} \left(1 - \frac{1}{M^2}\right)$.

With this choice of $(p_k)_{k \geq 1}$ and $(n_k)_{k \geq 1}$, consider $\hat{\mathbb{E}}(L_0) = \sum_k p_k n_k$.  Note that for each $n \in [n_{k-1}, n_k-1]$, we have $\var(n) \leq r_k/2$.  Thus, setting $n_0 =r_0= 1$,
\begin{align*}
\sum_{n \geq 1} \var(n)
	&\geq \sum_{k \geq 1} (n_k - n_{k-1}) \frac{r_k}{2}\\
	& = \frac{1}{2}\sum_{k \geq 1} n_k r_k - \frac{1}{2}\sum_{k \geq 1} n_{k-1} r_k\\
	&\geq \frac{1}{2} \sum_{k \geq 1} n_k r_k - \frac{1}{2}\left(1 - \frac{1}{M^2}\right)\sum_{k \geq 1} n_{k-1} r_{k-1}\\
	& = \frac{1}{2M^2} \sum_{k \geq 1} n_k r_k - \frac{1}{2} \left(1 - \frac{1}{M^2}\right)\\
	&\geq \frac{1}{2M^2} \sum_{k \geq 1} n_k p_k - \frac{1}{2}.
\end{align*}
Thus, using the fact that $\sum \var(n) < \infty$,
\[
\hat{\mathbb{E}}(L_0) = \sum_{k \geq 1} n_k p_k \leq 2M^2\left(1 + \sum_{n \geq 1} \var(n) \right) < \infty.
\]
This completes the proof of the theorem.
\end{proof}

\subsection{Another construction for finite alphabets}\label{subsec:determ_construction}

The following section contains a different construction to show that every random Markov process on a finite alphabet is a deterministic random Markov process.  The result is not as strong as Theorem \ref{thm:countable} and is independent of further results, and can be skipped, but is presented as an alternative approach.

In the construction given in the proof of Proposition \ref{prop:determ} below, the digits of the binary expansion of probabilities of the `present state' given the `past' are used to define a new deterministic random Markov process from any random Markov process on a finite alphabet.  A key tool in this proof is the use of an injective function $F: (\mathbb{Z}^+)^{n+1} \to \mathbb{Z}^+$ with the property that for any $i_0, i_1, \ldots, i_n$, $F(i_0, i_1, \ldots, i_n) \geq i_0$.  Roughly, conditioned on a particular past, the event $\{L_0 =i_0\}$ is decomposed into infinitely many parts $\{\hat{L}_0 = F(i_0, i_1, \ldots, i_n) \mid i_1, \ldots, i_n \in \mathbb{Z}^+\}$.  After the proof of Proposition \ref{prop:determ}, a method of constructing such functions is described along with some examples.

\begin{prop}\label{prop:determ}
Let $\left(\left(X_i, L_i\right)_{i \in \mathbb{Z}}, \mathbb{P}\right)$ be a random Markov process on a finite alphabet $A$, let $n \in \mathbb{Z}^+$ be such that $|A| \leq 2^n$, and let $F: (\mathbb{Z}^+)^{n+1} \to \mathbb{Z}^+$ be an injective function with the property that for every $i_0, i_1, \ldots, i_n \in \mathbb{Z}^+$, $F(i_0, i_1, \ldots, i_n) \geq i_0$.  There is an independent process $\left(\hat{L}_i\right)_{i\in\ZZ}$ with
\[
\hat{\mathbb{P}}\left(\hat{L}_0 = F\left(i_0, i_1, \ldots, i_n\right)\right) = \mathbb{P}\left(L_0 = i_0\right) \frac{1}{2^{i_1}}\cdots \frac{1}{2^{i_n}}
\]
and a stationary coupling so that $\left(\left(X_i, \hat{L}_i\right)_{i \in \mathbb{Z}}, \hat{\mathbb{P}}\right)$ is a deterministic complete random Markov process.
\end{prop}

\begin{proof}
Assume, without loss of generality, that $A = \{0,1\}^n$ and write $X_i = \left(Y_i^1, Y_i^2, \ldots, Y_i^n\right)$.  Given an integer $i$, a fixed sequence of past states $\omega_{-1}, \ldots, \omega_{-i}$, and $d_1, \ldots, d_n \in \left\{0,1\right\}$, let $\left\{\varepsilon_k^j \mid k \geq 1,\ j \in \left[1,n\right]\right\}$ be such that for every $k, j$, $\varepsilon_k^j \in \left\{0,1\right\}$ and
\begin{align*}
\mathbb{P}&(Y_0^1 = d_1 \mid (X_j)_{-i}^{-1} = (\omega_j)_{-i}^{-1} \wedge L_0 = i) 
	= \sum_{k \geq 1} \frac{\varepsilon_k^1}{2^k},\\
\mathbb{P}&(Y_0^2 = d_2 \mid (X_j)_{-i}^{-1} = (\omega_j)_{-i}^{-1} \wedge Y_0^1 = d_1 \wedge L_0 = i) 
	= \sum_{k \geq 1} \frac{\varepsilon_k^2}{2^k},\\
	\vdots\\
\mathbb{P}&(Y_0^n = d_n \mid (X_j)_{-i}^{-1} = (\omega_j)_{-i}^{-1} \wedge Y_0^1 = d_1, Y_0^2 = d_2, \ldots,\\\
	& \qquad  Y_0^{n-1} = d_{n-1} \wedge L_0 = i) \\
	&= \sum_{k \geq 1} \frac{\varepsilon_k^n}{2^k}.
\end{align*}

Thus, 
\begin{multline}
\mathbb{P}(Y_0^1 = d_1, \ldots, Y_0^n = d_n \mid (X_j)_{-i}^{-1} = (\omega_j)_{-i}^{-1} \wedge L_0 = i)\\
	 = \prod_{i=1}^n \left(\sum_{k \geq 1} \frac{\varepsilon_k^i}{2^k} \right) = \sum_{i_1, i_2, \ldots, i_n} \frac{\prod_{j=1}^n \varepsilon_{i_j}^j}{2^{i_1+i_2 + \cdots+i_n}}.
\end{multline}

For any $i_1, i_2, \ldots, i_n \geq 1$, define
\[
\hat{\mathbb{P}}\left(\hat{L}_0 = F\left(i, i_1, i_2, \ldots, i_n\right)\right) = \mathbb{P}\left(L_0 = i\right) \frac{1}{2^{i_1}}\cdot \frac{1}{2^{i_2}}\cdots \frac{1}{2^{i_n}}
\]
and set
\begin{multline}
\hat{\mathbb{P}}\left((Y_0^i)_{i=1}^n =  (d_i)_{i=1}^n \mid \left(X_j\right)_{j=-i}^{-1} = \left(\omega_j\right)_{j=-i}^{-1} \wedge \hat{L}_0 = F\left(i, i_1, i_2, \ldots, i_n\right)\right)\\
	= \prod_{k=1}^n \varepsilon_{i_k}^k \in \left\{0,1\right\}.
\end{multline}

Since the function $F$ is injective, $F\left(i, i_1, i_2, \ldots, i_n\right) > i$, and 
\[
\sum_{i_1, \ldots, i_n} \hat{\mathbb{P}}\left(\hat{L}_0 = F\left(i, i_1, \ldots, i_n\right)\right) = \mathbb{P}\left(L_0 = i\right),
\]
 this defines a deterministic complete random Markov process. 
\end{proof}

The proof of Proposition \ref{prop:determ} depends on the existence of the functions $F:(\mathbb{Z}^+)^n \to \mathbb{Z}^+$ that are both injective and have the property that for every choice of $i_0, i_1, \ldots, i_n \in \mathbb{Z}^+$, $F(i_0, i_1, \ldots, i_n) \geq i_0$.  Here, we describe one possible method of constructing such functions.

Let $\{B_i \mid i \in \mathbb{Z}^+\}$ be a collection of disjoint sets of integers with the property that for every $i$, $\min B_i \geq i$.  For example if $\{q_1 < q_2 < \cdots\}$ is the set of primes, then choosing the sets $B_i = \{q_i^k \mid k \geq 1\}$ will have the desired property.  For any such collection $\{B_i\}_{i \in \mathbb{Z}^+}$, a sequence of functions $\{F_n\}_{n\geq 1}$ is defined so that for each $n \geq 1$, the function $F_n: (\mathbb{Z}^+)^{n+1} \to \mathbb{Z}^+$ is injective and has the property that for every $i_0, i_1, \ldots, i_n$, $F_n(i_0, i_1, \ldots, i_n) \geq i_0$.  

To begin the recursive construction, for $n =1$ and $i, j \in \mathbb{Z}^+$, define $F_1(i,j)$ to be the $j$-th smallest element of $B_i$.  The function $F_1$ is injective since the sets $\{B_i\}_{i \in \mathbb{Z}^+}$ are all disjoint and by construction, $F_1(i, j) \geq \min B_i \geq i$.

For $n \geq 2$, given $F_{n-1}$ with the desired properties, define $F_n : (\mathbb{Z}^+)^{n+1} \to \mathbb{Z}^+$ as follows.  For $i_0, i_1, \ldots, i_n \in \mathbb{Z}$, let $F_n(i_0, i_1, \ldots, i_n)$ be the $i_n$-th smallest element of $B_{F_{n-1}(i_0, i_1, \ldots, i_{n-1})}$.  The function $F_n$ is injective since $F_{n-1}$ is injective and the sets $\{B_i\}_{i \geq 1}$ are disjoint.  Further,
\[
F_n(i_0, i_1, \ldots, i_n) \geq \min B_{F_{n-1}(i_0, i_1, \ldots, i_{n-1})} \geq F_{n-1}(i_0, i_1, \ldots, i_{n-1}) \geq i_0.
\]
This completes the recursive construction.

Note that using a function $F$ defined in terms of powers of primes, as above, will lead to a deterministic random Markov process from the proof of Proposition \ref{prop:determ} with $\mathbb{E}(\hat{L}_0) = \infty$ even if $\mathbb{E}(L_0) < \infty$.  However, if $((X_i, L_i)_i, \mathbb{P})$ is a random Markov process on a finite alphabet with $\mathbb{E}(L_0) < \infty$, then by Theorem \ref{thm:countable}, it is also a deterministic random Markov process with a finite expected look-back distance.  It turns out that a suitably chosen function $F$ can be used with Proposition \ref{prop:determ} to give an alternate proof of this fact. 

After describing a version of the proof of Proposition \ref{prop:determ} for 2 letter alphabets, together with the construction of functions above, to Paul Balister, he gave the following construction for sets of integers to show that a complete random Markov process on a two-letter alphabet with finite expected look-back distance is also a deterministic random Markov processes with the same property.  When the proof was generalized to arbitrary finite alphabets, it was realized that these sets together with the method of constructing injective functions described above could be used to prove the corresponding result for any finite alphabet.

For each $i \geq 1$, set $B_i^0 = \left\{4i-1\right\}$ and for each $n \geq 1$, recursively define
\[
B_i^n = \left\{4t+1 \mid t \in B_i^{n-1}\right\} \cup \left\{4t+2 \mid t \in B_i^{n-1}\right\}.
\]  
Set $B_i = \cup_{n \geq 0} B_i^n$.  For each $i\geq 1$, $\min B_i = 4i-1 > i$.  Arguing by congruence module $4$, for each $i \neq j$, then $B_i \cap B_j = \emptyset$.  Further, it can be shown, by induction, that for each $n \geq 0$, $B_i^n \subseteq [4^{n+1}\left(i-1\right), 4^{n+1}i]$ and also $|B_i^n| = 2^n$.  Thus, defining a function $F_1: (\mathbb{Z}^+)^2 \to \mathbb{Z}^+$, as above, in terms of these sets $\{B_i\}_{i \geq 1}$ has the property that
\[
\sum_{j \geq 1} \frac{F_1\left(i,j\right)}{2^j} \leq \sum_{n \geq 0} \sum_{\ell \in B_i^n} \frac{\ell}{2^{2^n}} \sum_{n \geq 0} \frac{2^n 4^{n+1}i}{2^{2^n}} \leq 35 i.
\]

Using the recursive definition of the functions $F_n$, one can show that for every $n \geq 1$ and $i_0 \in \mathbb{Z}^+$,
\[
\sum_{i_1, i_2, \ldots, i_n \in \mathbb{Z}^+} \frac{F_n(i_0, i_1, \ldots, i_n)}{2^{i_1+i_2+\cdots+i_n}} \leq (35)^n i.
\]

Therefore, using the construction in Proposition \ref{prop:determ} this generalizes to arbitrary finite alphabets and the sets $\left\{F\left(i\right) \mid i \geq 1\right\}$ can be used to  show that if $|A| \leq 2^n$, then there is a function $F$ that can be used to construct $\hat{L}$, with the property that
\[
\mathbb{E}\left(\hat{L}_0\right) \leq 35^n \cdot \mathbb{E}\left(L_0\right).
\]

\subsection{Examples}

In this section, some examples are given to show the limits of possible extensions of Theorem \ref{thm:countable}.  One can verify that all of the examples given in this section on finite and countable alphabets have finite entropy.  This is of interest because it is often the case that results that hold for processes on finite alphabets also hold for processes with finite entropy on countable alphabets.

The following example shows that without the assumption of a finite dominating measure, not all random Markov processes are deterministic random Markov processes. 

\begin{example}\label{ex:not-determ}
Consider the Markov process, $\left(X_i, Y_i\right)_{i \in \mathbb{Z}}$ defined on a countable state space $\mathbb{Z}^+ \times \mathbb{Z}^+$ with the first coordinate given by one Markov process and, for $n \geq 1$, given $X_0 = n$, the second coordinate is chosen independently in the integers $[1,n]$. Precisely, the $g$-function is defined by
\begin{equation}\label{eq:markov-not-determ1}
\mathbb{P}\left(X_0 = a, Y_0 = 1 \mid X_{-1} = b\right)=
	\begin{cases}
		 1,	&\text{if $a=1$, $b=0$}\\
		 \frac{2}{3}		&\text{if $a = 1$, $b=2$}\\
		 \frac{2}{3}		&\text{if $a = 0$, $b=1$}
	\end{cases}
\end{equation}
and for all $n \geq 2$ and $k \in [1,n]$,
\begin{align}
\mathbb{P}\left(X_0 = n, Y_0 = k \mid X_{-1} = n-1\right) &=\frac{1}{3n}, \notag\\
\mathbb{P}\left(X_0 = n, Y_0 = k \mid X_{-1} = n+1\right) &=\frac{2}{3n}. \label{eq:markov-not-determ2}
\end{align}
The stationary distribution for this process, denoted $\{\pi(n,k) \mid n \geq 0,\ k \in [1, \max\{1,n\}]\}$, is given by
\begin{equation}\label{eq:not-determ-stat-dist}
 \pi(n,k) = 
 \begin{cases}
    \frac{1}{4}	&\text{if } (n,k) = (0,1)\\
    \frac{3}{n 2^{n+2}}	&\text{if } n \geq 1 \text{ and } k \in [1,n].\\
 \end{cases}
\end{equation}

One can construct a stationary process with a $g$-function given by \eqref{eq:markov-not-determ1} and \eqref{eq:markov-not-determ2} directly by defining measures on cylinder sets.

Suppose that $\left(X_i, Y_i\right)_{i \in \mathbb{Z}}$ were a deterministic random Markov process and let $\left(L_i\right)_{i \in \mathbb{Z}}$ be a process of look-back distances.  Let $k$ be a positive integer so that $\mathbb{P}\left(L_0 = k\right) = p > 0$ and let $n \geq \lceil 1/p \rceil +1$.  Let  $\left(a_i, b_i\right)_{i=-k}^{-1}$ be such that $a_{-1} = n$ and $\mathbb{P}\left(\left(X_i, Y_i\right)_{-k}^{-1} = \left(a_i, b_i\right)_{i=-k}^{-1}\right)>0$.  Then, given that the process is a deterministic random Markov process, there exist $a_0, b_0$ such that 
\[
\mathbb{P}\left(X_0 = a_0, Y_0 = b_0 \mid \left(X_i, Y_i\right)_{-k}^{-1} = \left(a_i, b_i\right)_{i=-k}^{-1} \wedge L_0 = k\right) = 1
\]
 and so
\begin{align*}
\mathbb{P}&\left(X_0 = a_0, Y_0 = b_0 \mid \left(X_i, Y_i\right)_{i=-k}^{-1} = \left(a_i, b_i\right)_{i=-k}^{-1}\right) \\
	&\geq \mathbb{P}\left(X_0 = a_0, Y_0 = b_0 \wedge L_0 = k \mid \left(X_i, Y_i\right)_{i=-k}^{-1} = \left(a_i, b_i\right)_{i=-k}^{-1}\right)\\
	&= p \cdot \mathbb{P}\left(X_0 = a_0, Y_0 = b_0 \mid \left(X_i, Y_i\right)_{i=-k}^{-1} = \left(a_i, b_i\right)_{i=-k}^{-1} \wedge L_0 = k\right).
\end{align*}
However, since $a_{-1} = n \geq \lceil 1/p \rceil + 1$, 
\[
\mathbb{P}\left(X_0 = a_0, Y_0 = b_0 \mid \left(X_i, Y_i\right)_{i=-k}^{-1} = \left(a_i, b_i\right)_{i=-k}^{-1}\right) \leq \frac{2}{3\left(n-1\right)} < p.
\]
Thus, $\left(\left(X_i, Y_i\right)_{i \in \mathbb{Z}}, \mathbb{P}\right)$ is not only a random Markov process, but a usual Markov process on a countable alphabet, and hence a random Markov process with a finite expected look-back distance, that is not a deterministic random Markov process.  On can easily check that $\left(\left(X_i, Y_i\right)_{i \in \mathbb{Z}}, \mathbb{P}\right)$ has no finite dominating measure.

\qed \end{example}

In Theorem \ref{thm:countable}, it was shown that a uniform martingale on a finite alphabet has a representation as a deterministic random Markov process with finite expected look-back distance if{f} $\sum \var_n (\mathbb{P}) < \infty$.  As shown in the proof, even if the alphabet is countable, any deterministic random Markov process with finite expected look-back distance satisfies $\sum \var_n (\mathbb{P}) < \infty$.  As the following example shows, the reverse implication is not, in general, true, even for a process with a finite dominating measure.

\begin{example}\label{ex:inf-look-back}
Let $\left\{Z_i\right\}_{i \geq 1}$ be a collection of disjoint sets so that for every $i \geq 1$, $|Z_i| = 4^i$.  

Let $A = \cup_{i=1}^{\infty} Z_i$ be the countable alphabet and define an independent process $\left\{\left(X_i\right)_{i \in \mathbb{Z}}\right\}$ with the following stationary measure.  For every $i \geq 1$ and $a \in Z_i$, set
\[
\mathbb{P}\left(X_0 = a \right)= \frac{1}{2^i|Z_i|}.
\]
As this is an independent process, the stationary measure on $X_0$ is a finite dominating measure and for every $n \geq 1$, $\var_n (\mathbb{P}) = 0$ and hence $\sum_n \var_n (\mathbb{P}) = 0 < \infty$.  By Theorem \ref{thm:countable}, this process has a representation as a deterministic random Markov process. Note that any independent process is a random Markov process with a finite expected look-back distance. 

This process cannot, however, be represented as a deterministic random Markov process with finite expected look-back distance. To see this, let $(L_i)_{i \in \mathbb{Z}}$ be any sequence of look-back distances with which $\left(X_i\right)_{i \in \mathbb{Z}}$ is a deterministic random Markov process.  Fix $\ell \geq 1$ and let $k \geq \ell$ be the smallest integer such that $\mathbb{P}\left(L_0 > k\right) < \frac{1}{100 \cdot 2^\ell}$.

Fix any $a_1, a_2, \ldots, a_k \in A$ and consider the probability that $X_0 \in Z_\ell$ conditioned on the event that for every $i \in [1,k]$, $X_{-i} = a_i$.  Then, as the process is independent,
\begin{align*}
\frac{1}{2^\ell} = \mathbb{P}(X_0 \in Z_{\ell}) &= \mathbb{P}\left(X_0 \in Z_\ell \mid X_{-1} = a_1, X_{-2} = a_2, \ldots, X_{-k} = a_k\right)\\
	&= \mathbb{P}\left(X_0 \in Z_{\ell}, L_0 \leq k \mid X_{-1} = a_1, X_{-2} = a_2, \ldots, X_{-k} = a_k\right) \\
	&	\qquad + \mathbb{P}\left(X_0 \in Z_{\ell}, L_0 > k \mid X_{-1} = a_1, X_{-2} = a_2, \ldots, X_{-k} = a_k\right)\\
	&\leq \mathbb{P}\left(X_0 \in Z_{\ell}, L_0 \leq k \mid X_{-1} = a_1, X_{-2} = a_2, \ldots, X_{-k} = a_k\right)\\
		&\qquad + \frac{1}{2^{\ell} 100}.
\end{align*}

Thus, 
\[
\mathbb{P}\left(X_0 \in Z_\ell, L_0 \leq k \mid X_{-1} = a_1, X_{-2} = a_2, \ldots, X_{-k} = a_k\right) \geq \frac{0.99}{2^\ell}.
\]
As this is a deterministic random Markov process, for each $i \leq k$, there is at most one element, $a \in Z_\ell$ for which the probability that $X_0 = a$ conditioned on the fixed past and the event $L_0 = i$ is positive.  Define the set 
\begin{multline*}
A_{\ell, k} = \big\{z \in Z_\ell \mid \exists\ i \leq k \text{ s.t. }\\ \mathbb{P}\left(X_0 = z \mid L_0 = i,  X_{-1} = a_1, X_{-2} = a_2, \ldots, X_{-i} = a_i\right) = 1\big\}.
\end{multline*}
Since $a_1, a_2, \ldots, a_k$ are fixed, then since $\mathbb{P}$ is a probability measure, $|A_{\ell, k}| \leq k$ and
\begin{align}
\frac{0.99}{2^{\ell}} &\leq \mathbb{P}\left(X_0 \in Z_\ell, L_0 \leq k \mid X_{-1} = a_1, X_{-2} = a_2, \ldots, X_{-k} = a_k\right) \notag\\
	&= \mathbb{P}\left(X_0 \in A_{\ell, k}, L_0 \leq k \mid X_{-1} = a_1, X_{-2} = a_2, \ldots, X_{-k} = a_k\right) \notag\\
	&\leq \mathbb{P}\left(X_0 \in A_{\ell, k} \mid X_{-1} = a_1, X_{-2} = a_2, \ldots, X_{-k} = a_k\right) \notag\\
	&= \mathbb{P}\left(X_0 \in A_{\ell, k}\right) \notag\\
	&=\frac{|A_{\ell,k}|}{2^\ell |Z_\ell|}
	\leq \frac{k}{2^{\ell}|Z_\ell|}. \label{eq:z-prob}
\end{align}

Thus, by the inequalities in \eqref{eq:z-prob}, $k \geq 0.99|Z_\ell|$ and by the choice of $k$, 
\[
\mathbb{P}\left(L_0 > k-1\right) = \mathbb{P}\left(L_0 \geq k\right) \geq \frac{1}{100\cdot 2^{\ell}}.
\]
In particular,
\[
\mathbb{E}\left(L_0\right) \geq k \mathbb{P}\left(L_0 \geq k\right) \geq \frac{0.99 \cdot |Z_\ell|}{100\cdot 2^\ell} = \frac{0.99 \cdot 4^\ell}{100\cdot 2^\ell} =\frac{99\cdot 2^\ell}{10^4}. 
\]
As $\ell$ was arbitrary, $\mathbb{E}\left(L_0\right) = \infty$.
\qed \end{example}

One of the key questions that Theorem \ref{thm:countable} seeks to answer is which uniform martingales on countably infinite alphabets are random Markov processes.  The following example shows that, in this respect, the results in Theorem \ref{thm:countable} are best possible.  Example \ref{sec:rwbins} below is a uniform martingale on a countable alphabet that does not have a dominating measure and is not a random Markov process.

Consider the following example, showing that the properties of uniform martingales are not strong enough in general to guarantee that a uniform martingale is a random Markov process, without the additional assumption of a finite dominating measure.  This key example shows that Theorem \ref{thm:countable} is, in a strong sense, best possible.
  
  \begin{example}\label{sec:rwbins}
  The stochastic process presented here is similar to Example \ref{ex:not-determ} and is constructed as a joint distribution on pairs, where the first coordinate forms a stationary random walk on $\mathbb{N}$.  The distribution of the second coordinate is given in terms of the past values of the first coordinate.

Let $\left(B_n\right)_{n \in \mathbb{Z}}$ be the biased random walk on $\mathbb{N}$ given, for $i \geq 2$ and any $n \in \mathbb{Z}$ by
\begin{align*}
\mathbb{P}&\left(B_n = i+1 \mid B_{n-1}=i\right)	=1/3,\\
\mathbb{P}&\left(B_n = i-1 \mid B_{n-1} = i\right)	=2/3, \text{ and}\\
\mathbb{P}&\left(B_n = 2 \mid B_{n-1}=1\right)	=1.
\end{align*}
Let $\left\{\pi\left(i\right)\right\}_{i \in \mathbb{N}}$ be the stationary distribution for the process $\left(B_n\right)_{n \in \mathbb{Z}}$.  Then, similarly to Example \ref{ex:not-determ},
\[
\pi\left(i\right) = 
	\begin{cases}
		\frac{1}{4}	&\text{ if } i=1,\\
		\frac{3}{2^{i+1}}	&\text{ if } i \geq 2.
	\end{cases}
\]
The process $\left(B_n\right)_{n \in \mathbb{Z}}$ is used to define another process $\left(Y_n\right)_{n \in \mathbb{Z}}$ such that, conditioned on the event $B_n =i$, then $Y_n \in [1, i+1]$.  On can think of the random variable $B_n$ indicating a `bin' and $Y_n$ the position chosen within the $B_n$-th bin.  The process given in this example is defined so that, for every $k \geq 1$, 
\begin{equation}\label{eq:bins_cond}
\mathbb{P}(Y_0 = i \mid B_0 = B_{-2k} = k, Y_{-2k} = j) = 
	\begin{cases}
		\frac{1}{k}	&\text{if $i \neq j$,}\\
		0			&\text{otherwise}.\\
	\end{cases}  
\end{equation}
In all other cases, the state for $Y_0$ is chosen independently at random. Note that the state of $B_0$ does not depend on the sequence $(Y_i)_{i< 0}$.  To be precise, the $g$-function in question is defined by the following conditional probabilities.  For $\{k_n\}_{n \leq 0}$ and $\{j_n\}_{n \leq 0}$ such that for every $n \leq 0$, $j_n \in \{1, 2, \ldots, k_n+1\}$, then
\begin{multline}\label{eq:bins-g-function}
 \mathbb{P}\left(B_0 = k_0, Y_0 = j_0 \mid (B_n)_{n < 0} = (k_n)_{n < 0}, (Y_n)_{n < 0} = (j_n)_{n < 0} \right)\\
 = \begin{cases}
    \frac{2}{3(k_0 + 1)}	&\text{if } k_{-1} \geq 2, k_{-1} = k_0 - 1, \text{ and, } k_{0} \neq k_{-2k_0}\\
    \frac{2}{3k_0}		&\text{if } k_{-1} \geq 2, k_{-1} = k_0 - 1, k_{0} = k_{-2k_0}, \text{ and, } j_{0} \neq j_{-2k_0}\\
    \frac{1}{3(k_0+1)}		&\text{if } k_{-1} \geq 2, k_{-1} = k_0 + 1, \text{ and, } k_{0} \neq k_{-2k_0}\\
    \frac{1}{3k_0}		&\text{if } k_{-1} \geq 2, k_{-1} = k_0 + 1, k_{0} = k_{-2k_0}, \text{ and, } j_{0} \neq j_{-2k_0}\\
    \frac{1}{3}			&\text{if } k_0 = 2, k_{-1} = 1, \text{ and, } k_{-4} \neq 2\\
    \frac{1}{2}			&\text{if } k_0 = 2, k_{-1} = 1, k_{-4} = 2, \text{ and, } j_0 \neq j_{-4}\\
    0				&\text{otherwise}.\\
    \end{cases}
\end{multline}

The measure is constructed by recursively defining measures on cylinder sets, using the recursive technique described in the introduction.  Let $\mathbb{P}_0$ be the measure for the stationary process given by the random walk $\left(B_n\right)_{n \in \mathbb{Z}}$.

For every $k \geq 1$, $\mu_k$ will be a measure on $(\mathbb{N} \times \mathbb{N})^k$ that satisfies the $g$-function \ref{eq:bins-g-function} and with the property that for every $k \geq 0$, if $\min \{a_0, a_1, \ldots, a_{2k}\} > k$ and $j_0, j_1, \ldots, j_{2k}$ are such that for each $\ell \in [0, 2k]$, $j_\ell \in [1, a_{\ell} + 1]$, then 
\begin{multline}\label{eq:bins-pos-measure}
\mu_{2k+1}\left(B_0 = a_0, Y_0 = j_0, B_{1} = a_1, Y_{1} = j_1, \ldots, B_{2k} = a_{2k}, Y_{2k} = j_{2k} \right)\\
 = \mathbb{P}_0\left(B_0 = a_0, B_{1} = a_1, \ldots, B_{2k} = a_{2k} \right) \prod_{i=0}^{2k} \frac{1}{a_i + 1}.
\end{multline}

In particular, if for each $\ell \in [1, 2k]$, $|a_{\ell} - a_{\ell-1}| = 1$, then in particular,
\[
\mu_{2k+1}\left(B_0 = a_0, Y_0 = j_0, B_{1} = a_1, Y_{1} = j_1, \ldots, B_{2k} = a_{2k}, Y_{2k} = j_{2k} \right)>0.
\]

Define the measure $\mu_1$ so that for any $i \geq 1$ and $j \in [1, i+1]$,
\[
\mu_1(B_1 = i, Y_1 = j) = \mathbb{P}_0(B_1 = i) \frac{1}{i+1}.
\]
Define the measure $\mu_2$ for $i_1, i_2 \geq 1$, $j_1 \in [1, i_1+1]$ and $j_2 \in [1, i_2+1]$ by
\[
\mu_2(B_1 = i_1, Y_1 = j_1, B_2 = i_2, Y_2 = j_2) = \mathbb{P}_0(B_1 = i_1, B_2 = i_2) \frac{1}{i_1+1}\cdot \frac{1}{i_2 + 1}.
\]
Note that the condition \eqref{eq:bins-pos-measure} is trivially satisfied for $\mu_1$.

For each $k \geq 1$, given $\mu_{2k}$ and $\mu_{2k-1}$, define $\mu_{2k+1}$ as follows.  For each $\ell \in [1, 2k+1]$, let $i_{\ell} \geq 1$ and $j_{\ell} \in [1, i_{\ell}+1]$.  Define the following events:
\begin{align*}
 E_1 &=\{B_1 = i_1, Y_1 = j_1\},\\
 E_2 &=\{B_2 = i_2, Y_2 = j_2, \ldots, B_{2k} = i_{2k}, Y_{2k} = j_{2k}\},\\
 E_3 &=\{B_{2k+1} = i_{2k+1}, Y_{2k+1} = j_{2k+1}\}
\end{align*}
Define the measure of the cylinder set $E_1 \cap E_2 \cap E_3$ as follows.  If either $i_1 \neq k$ or $i_{2k+1} \neq k$, then define
\[
\mu_{2k+1}\left(E_1 \cap E_2 \cap E_3\right) = \mu_{2k}(E_1 \mid E_2) \mu_{2k-1}(E_2) \mu_{2k}(E_3 \mid E_2).
\]
Note that if \eqref{eq:bins-pos-measure} is satisfied for $\mu_{2k-1}$, then it is also satisfied for $\mu_{2k+1}$.

If $i_1 = i_{2k+1} = k$ and $j_1 \neq j_{2k+1}$, then define
\begin{multline*}
\mu_{2k+1}\left(E_1 \cap E_2 \cap E_3\right)\\ = \mathbb{P}_0(B_1 = i_1 \mid B_2 = i_2) \mu_{2k-1}(E_2) \mathbb{P}_0(B_{2k+1} = i_{2k+1} \mid B_{2k} = i_{2k}) \frac{1}{k^2+k}.
\end{multline*}
Otherwise, if $i_1 = i_{2k+1} = k$ and $j_1 = j_{2k+1}$, then define $\mu_{2k+1}\left(E_1 \cap E_2 \cap E_3\right) = 0$.

The measure $\mu_{2k+2}$ is defined by a single case.  Define the events
\begin{align*}
E_1 &=\{B_1 = i_1, Y_1 = j_1\},\\
 E_2 &=\{B_2 = i_2, Y_2 = j_2, \ldots, B_{2k+1} = i_{2k+1}, Y_{2k+1} = j_{2k+1}\},\\
 E_3 &=\{B_{2k+2} = i_{2k+2}, Y_{2k+2} = j_{2k+2}\}
\end{align*}
and define the measure $\mu_{2k+2}$ on the cylinder set $E_1 \cap E_2 \cap E_3$ by
\[
\mu_{2k+2}(E_1 \cap E_2 \cap E_3) = \mu_{2k+1}(E_1 \mid E_2)\mu_{2k}(E_2) \mu_{2k+1}(E_3 \mid E_2).
\]

The measures $(\mu_k)_{k \geq 1}$ are a consistent sequence of measures that define a stationary process on $(\mathbb{N} \times \mathbb{N})^{\mathbb{Z}}$ with the property given by equation \eqref{eq:bins_cond}.  The process is a uniform martingale with $n$-th variation satisfying
\[
\var(n) \leq \frac{2}{n}.
\]

Denote the measure of this process $\mathbb{P}$ and suppose, in hopes of a contradiction, that $((\mathbb{N} \times \mathbb{N})^{\mathbb{Z}}, \mathbb{P})$ were a random Markov process with look-back distance $(L_i)_{i \in \mathbb{Z}}$.  Fix $k$ with the property that $\mathbb{P}(L_0 = k) > 0$.  

Let $n \geq 2k$ and consider the event 
\[
E_n = 
	\begin{cases}
		\{B_{-1} = n+1, B_{-2} = n+2, \ldots, B_{-k} = n+1\}	&\text{$k$ odd}\\
		 \{B_{-1} = n+1, B_{-2} = n+2, \ldots, B_{-k} = n+2\}	&\text{$k$ even}.
	\end{cases}
\]
Then, the event $F_n = E_n \cap \{Y_{-1} = \cdots = Y_{-k} = 1\}$ has positive measure by equation \eqref{eq:bins-pos-measure}.  Conditioned on the event $F_n$, either $B_0 = n$ or $B_0 = n+2$ with probability $1$.  For every $j \in [1,n+1]$, there is a past, contained in $F_n$ with $B_{-2n} = n$ and $Y_{-2n} = j$.  Thus,
\begin{multline*}
\mathbb{P}(B_{0} = n, Y_{0} = j, L_0 = k \mid F_n)\\ = \mathbb{P}(B_{0} = n, Y_{0} = j, L_0 = k \mid F_n, B_{-2n} = n, Y_{-2n} = j) = 0.
\end{multline*}
Similarly, for every $j \in [1,n+3]$,
\begin{multline*}
\mathbb{P}(B_{0} = n+2, Y_{0} = j, L_0 = k \mid F_n)\\ = \mathbb{P}(B_{0} = n, Y_{0} = j, L_0 = k \mid F_n, B_{-2n-2} = n+2, Y_{-2n-2} = j) = 0.
\end{multline*}
Then,
\begin{align*}
\mathbb{P}(L_0 = k \mid F_n)
	&=\sum_{j=1}^{n+1} \mathbb{P}(B_{0} = n, Y_{0} = j, L_0 = k \mid F_n) +\\
	& \qquad \sum_{\ell =1}^{n+3} \mathbb{P}(B_{0} = n+2, Y_{0} = j, L_0 = k \mid F_n)\\
	&=0,
\end{align*}
contradicting the assumption that $\mathbb{P}(L_0 = k) > 0$, since $\mathbb{P}(F_n) > 0$ and the events $F_n$ and $\{L_0 = k\}$ are independent.

The details are not given here, but one could prove that the measure constructed in this example is reversible.

Therefore, this process is not a random Markov process.
\qed \end{example}

Example \ref{sec:rwbins} shows that something stronger than the assumptions in the definition of a uniform martingale is needed to guarantee that a stationary process is a random Markov process. However,  random Markov processes do not obey the dominating measure condition (Definition \ref{def:dommeasure}) in general; while every random Markov chain on a finite state space has a finite dominating measure, this need not be the case when the alphabet is infinite. We illustrate this by the following much simpler example.

\begin{example}
\label{example:rmnodom}
Consider the following Markov chain on the state space $\mathbb{Z}^+$.  Define the transition probabilities, for each $n \geq 1$ by
\begin{align*}
	\mathbb{P}\left(X_0 = 1 \mid X_{-1} = n\right) &=\frac{1}{2}\\
	\mathbb{P}\left(X_0 = n+1 \mid X_{-1} = n\right)	&=\frac{1}{2}.
\end{align*}
Then, $\left(X_i\right)_{i \in \mathbb{Z}}$ has a stationary probability measure with $\mathbb{P}\left(X_0 = n\right) = \frac{1}{2^n}$.  As $\left(\left(X_i\right)_{i \in \mathbb{Z}}, \mathbb{P}\right)$ is a Markov chain with a stationary probability measure, this process is also a random Markov chain.  However, there can be no finite dominating measure.  
\qed \end{example}

\section{Uncountable alphabets}\label{sec:unctble}

Example \ref{sec:rwbins} in the previous section demonstrates that the notions of uniform martingale and random Markov process are not equivalent when the alphabet for the process is infinite.  In this section, another condition is considered which implies that a process is a random Markov process, even  for uncountable alphabets. 

Recall that a stationary process $\left(\left(X_n\right)_{n \in \mathbb{Z}}, \mathbb{P}\right)$, on alphabet $A$, satisfies  \emph{Berbee's ratio condition} (see Definition \ref{cond:four}) if{f} for every $\varepsilon >0$, there is an $n$ so that for every $\mathbf{\omega} \in A^{\mathbb{Z}^-}$ and $E_0$, a measurable subset of $A$,
\begin{equation}\label{eq:ratio}
\left|\frac{\mathbb{P}\left(X_0 \in E_0 \mid X_{-1} = \omega_{-1}, \ldots\right)}{\mathbb{P}\left(X_0 \in E_0 \mid X_{-1} = \omega_{-1}, \ldots, X_{-n} = \omega_{-n}\right)}  - 1 \right| < \varepsilon.
\end{equation} 

The above choice of ratio is chosen so that if the denominator is $0$ then the numerator is also.  Thus, adopting the convention that $\frac{0}{0} = 1$, the fraction in \eqref{eq:ratio} is well-defined for all $A$ and $\mathbf{\omega}$.

While any stationary process that satisfies the ratio condition is certainly a uniform martingale, the converse need not be true, in general. Indeed, we now give a series of examples illustrating the differences between these notions. 

First, the idea behind Example \ref{sec:rwbins} is modified to explicitly construct an example of a uniform martingale on an uncountable alphabet that has a dominating measure, is not a random Markov process and does not satisfy the ratio condition.  Note that by Theorem \ref{thm:ratio}, it must be the case that any such stationary process does not satisfy Berbee's ratio condition. 

\begin{example}\label{ex:unctble-notrm}
Let $A = \left(0,1\right]$ be the half-open unit interval and let $\lambda$ be the usual Lebesgue measure on $A$.  The example given here is a uniform martingale on $A$, with stationary measure $\lambda$ and finite dominating measure $2\cdot \lambda$ that is not a random Markov process. 

For every $x \in (0,1]$, there is a unique pair $(i, r)$ with $i \geq 1$ and $r \in (0,1]$ so that $x = \frac{1}{2^i}(1+r)$.  The alphabet $A$ is treated equivalently as $(0,1]$ and as $\mathbb{Z}^+ \times (0,1]$ and the values of $i$ and $r$ are used below to simplify the definition of the process and $g$-function.  On $\mathbb{Z}^+ \times (0,1]$, Lebesgue measure $\lambda$ can be given by choosing $i$ and $r$ independently, with $\mathbb{P}(i = i_0) = \frac{1}{2^{i_0}}$ and $r$ chosen according to $\lambda$.

For every $k \geq 1$ and $j \in \{1, 2, \ldots, k\}$, define
\begin{equation}\label{eq:small-ints}
I_{k,j} = \bigg(\frac{j-1}{k}, \frac{j}{k} \bigg]
\end{equation}
These sets will play the role of the `bins' as in Example \ref{sec:rwbins}.

The stationary process $((X_n, R_n)_{n \in \mathbb{Z}}, \mathbb{P})$ constructed in this section has conditional probabilities given as follows.  For every $(i_n, r_n)_{n < 0} \subseteq (\mathbb{Z}^+ \times (0,1])^{\mathbb{Z}^-}$, $i_0 \in \mathbb{Z}^+$,  and measurable set $E_0 \subseteq (0,1]$,
\begin{multline}\label{eq:unctble-g-fn}
\mathbb{P}\left(X_0 = i_0, R_0 \in E_0 \mid X_{-1} = i_{-1}, R_{-1} = r_{-1}, X_{-2} = i_{-2}, R_{-2} = r_{-2}, \ldots \right)\\
=
\begin{cases}
	\frac{1}{2^{i_0}} \frac{\lambda(E \cap I_{k, \ell}^c)}{1 - 1/k}	&\text{if $i_{-1} = \cdots = i_{-k+1} = 1$, $i_{-k} = k$, and $r_{-k} \in I_{k, \ell}$}\\
	\frac{1}{2^{i_0}} \lambda(E)	&\text{otherwise}.
\end{cases}
\end{multline}
That is, the values of $i_0$ are always $\ell$ with probability $2^{-\ell}$ and $r_0$ is chosen according to Lebesgue measure unless $i_{-1} = \cdots = i_{-k+1} = 1$ and $i_{-k} = k$, in which case, $r_0$ is chosen in a bin different from the bin containing $r_{-k}$.

To see that any stationary process that satisfies equation \eqref{eq:unctble-g-fn} is a uniform martingale (Definition \ref{def:um}), fix $(i_j, r_j)_{j < 0} \subseteq (\mathbb{Z}^+ \times (0,1])^{\mathbb{Z}^-}$, $i_0 \in \mathbb{Z}^+$, $n \geq 1$ and let $E_0 \subseteq (0,1]$ be a measurable set.  If for some $m \geq n+1$, $i_{-1} = \cdots = i_{-m+1} = 1$, then for some $\ell \in \{1, 2, \ldots, m\}$,
\begin{align*}
\vert \mathbb{P}&\left(X_0 = i_0, R_0 \in E_0 \mid X_{-1} = i_{-1}, R_{-1} = r_{-1}, \ldots \right)\\
& - \mathbb{P}\left(X_0 = i_0, R_0 \in E_0 \mid X_{-1} = i_{-1}, R_{-1} = r_{-1}, \ldots, X_{-n} = i_{-n}, R_{-n} = r_{-n} \right) \vert\\
	&\leq \frac{1}{2^{i_0}}\bigg\vert \frac{\lambda(E_0 \cap I_{m, \ell}^c)}{1 - 1/m} - \lambda(E_0)\bigg\vert\\
	& = \frac{1}{2^{i_0}}\bigg\vert \frac{\lambda(E_0 \cap I_{m, \ell}^c)}{m-1} - \lambda(E_0 \cap I_{m, \ell})\bigg\vert\\
	&\leq \frac{1}{2}\left(\frac{1 - 1/m}{m-1} + \frac{1}{m} \right)\\
	& = \frac{1}{m} \leq \frac{1}{n+1}.
\end{align*}
Otherwise,
\begin{multline*}
\mathbb{P}\left(X_0 = i_0, R_0 \in E_0 \mid X_{-1} = i_{-1}, R_{-1} = r_{-1}, \ldots \right) =\\ \mathbb{P}\left(X_0 = i_0, R_0 \in E_0 \mid X_{-1} = i_{-1}, R_{-1} = r_{-1}, \ldots, X_{-n} = i_{-n}, R_{-n} = r_{-n} \right)
\end{multline*}

Further, to see that any stationary process satisfying equation \eqref{eq:unctble-g-fn} has $2 \cdot \lambda$ as a finite dominating measure, note that for any $m \geq 1$, $\ell \in \{1, 2, \ldots, m\}$,
\[
\frac{\lambda(E \cap I_{m, \ell}^c)}{1 - 1/k} \leq 2 \lambda(E).
\]

Rather than appealing to fixed-point theorems to show that there is at least one stationary process satisfying equation \eqref{eq:unctble-g-fn}, such a measure can be explicitly constructed by recursion in order to guarantee that certain natural sets have positive measure.  Some details of this construction are given here and closely follow the construction method used in Example \ref{sec:rwbins}.

A sequence of consistent measures $(\mu_k)_{k \geq 1}$ is constructed so that for each $k \geq 1$, $\mu_k$ is a measure on $(\mathbb{Z}^+ \times (0,1])^k$ and with elements of $(\mathbb{Z}^+ \times (0,1])^{k}$ defined by $X_1, \ldots, X_k \in \mathbb{Z}^+$ and $R_1, \ldots, R_k \in (0,1]$.  The measures constructed will have the property that for every $\ell \leq k-1$, $i_1, i_2, \ldots, i_k$ does not contain a subsequence of length $\ell$ of the form $\ell, 1, 1, \ldots, 1$, then
\begin{multline}\label{eq:no-bad-event}
\mu_k(X_1 = i_1, R_1 \in E_1, X_2 = i_2, R_2 \in E_2, \ldots, X_k = i_k, R_k \in E_k)\\ = \frac{1}{2^{i_1 + i_2 + \cdots + i_k}} \prod_{j = 1}^{k} \lambda(E_j).
\end{multline}
That is, except for certain special choices of $i_1, i_2, \ldots, i_k$, the measure $\mu_k$ will couple the measure on each coordinate independently.

To start the recursion, let $\mu_1 \sim \lambda$ be the usual Lebesgue measure and let $\mu_2 \sim \lambda \times \lambda$ be the usual product measure with $\lambda$.  Both measures $\mu_1$ and $\mu_2$ satisfy condition \eqref{eq:no-bad-event} trivially.

For the recursion step, fix $k \geq 3$ and suppose that $\mu_{k-1}$ and $\mu_{k-2}$ have been defined and satisfy the condition in equation \eqref{eq:no-bad-event}.  Fix $i_1, i_2, \ldots, i_k \in \mathbb{Z}^+$ and let $E_1, E_2, \ldots E_k \subseteq (0,1]$ be measurable sets.  Without loss of generality, assume that for every $j \in [1,k]$, there exists $\ell \in [1,k-1]$ so that $E_j \subseteq I_{k-1, \ell}$.  The measure can subsequently be defined on arbitrary sets by taking finite disjoint unions.  Define the events
\begin{align*}
F_1 &= \left\{X_1= i_1, R_1 \in E_1 \right\}\\
F_{k-2} &=\bigg\{X_2 = i_2, R_2 \in E_2, X_3 = i_3, R_3 \in E_3, \ldots, X_{k-1} = i_{k-1}, R_{k-1} \in E_{k-1}\bigg\}\\
F_k &= \left\{ X_k = i_k, R_{k} \in E_k\right\}\\
\end{align*}

In order to define $\mu_k(F_1 \cap F_{k-2} \cap F_k)$, the measures $\mu_{k-2}, \mu_{k-1}$ are used to define an coupling of $\mu_{k-1}(F_1 \cap F_{k-2})$ and $\mu(F_{k-2} \cap F_k)$ in a way that satisfies equation \eqref{eq:no-bad-event} and condition \eqref{eq:unctble-g-fn}.

If $i_1 = k-1$ and $i_2 = \cdots = i_{k-1} = 1$, then let $\ell_1, \ell_2$ be such that $E_1 \subseteq I_{k-1, \ell_1}$ and $E_k \subseteq I_{k-1, \ell_2}$.  Define
\begin{equation}\label{eq:muk1}
\mu_k\left(F_1 \cap F_{k-2} \cap F_k \right)
	 = \mu_{k-2}(F_{k-2}) \cdot
	\begin{cases}
		\frac{1}{2^{i_k}} \frac{1}{2^{i_1}} \lambda(E_1) \lambda(E_k) \cdot \frac{k-1}{k-2}	&\text{if $\ell_1 \neq \ell_2$}\\
		0	&\text{if $\ell_1 = \ell_2$}.
	\end{cases}
\end{equation} 
Otherwise, define
\begin{equation}\label{eq:muk2}
\mu_k\left(F_1 \cap F_{k-2} \cap F_k \right)
	 = 
		\mu_{k-2}\left(F_{k-2}\right) \cdot \mu_{k-1}\left(F_{1} \mid F_{k-2} \right) \cdot \mu_{k-1}\left(F_{k} \mid F_{k-2} \right). 
\end{equation}

Using the Carath\'{e}odory extension theorem, the measure $\mu_k$ is extended from cylinder sets to all measurable sets.

One can show that the measure $\mu_k$ defined by equations \eqref{eq:muk1} and \eqref{eq:muk2} satisfy conditions \eqref{eq:unctble-g-fn} and \eqref{eq:no-bad-event} and in particular, for every $k, \ell$, 
\[
\mathbb{P}\left(X_{-k} = k, X_{-k+1} = 1, \cdots, X_{-1} = 1, X_0 = \ell \right) > 0.
\]

This completes the construction of the sequence of consistent measures $(\mu_k)_{k \geq 1}$ that are stationary by construction.  Thus, by the Kolmogorov extension theorem, there is a measure $\mathbb{P}$ on $(\mathbb{Z}^+ \times (0,1])^\mathbb{Z}$ with marginals given by the measures $(\mu_k)_{k\geq 1}$ that is stationary by construction and satisfies conditions \eqref{eq:unctble-g-fn} and \eqref{eq:no-bad-event}.

To see that this uniform martingale is not a random Markov process, suppose, in hopes of a contradiction that it were and let $k$ be such that $\mathbb{P}(L_0 = k) > 0$.  Then, for $r_{-1}, r_{-2}, \ldots, r_{-k+1} \in (0,1]$ and every $j \geq 1$ and $E_0 \in (0,1]$ be measurable,
\begin{align*}
\mathbb{P}&\bigg(X_0 = j, R_0 \in E_0 \wedge L_0 = k \mid X_{-1} = 1, R_{-1} = r_1, \ldots,\\
	&\qquad X_{-k+1} = 1, R_{-k+1} = r_{-k+1}  \bigg)\\
	&=\mathbb{P}\bigg(X_0 = j, R_0 \in E_0 \wedge L_0 = k \mid X_{-1} = 1, R_{-1} = r_1, \ldots,\\
	&\qquad X_{-k+1} = 1, R_{-k+1} = r_{-k+1}, X_{-k} = k, R_{-k} = j  \bigg)\\
	& = 0.
\end{align*}

As $j, r$ and $r_{-1}, \ldots, r_{-k+1}$ were arbitrary,
\[
\mathbb{P}(L_0 = k \mid X_{-1} = X_{-2} = \cdots = X_{-k+1} = 1) = 0,
\]
contradicting the assumption that these events are independent and occur with positive probability.  Therefore, the process is not a random Markov process.
\qed \end{example} 

While a stationary process that satisfies the ratio condition is a uniform martingale, it need not have a dominating measure, as the following examples show.

\begin{example}
Two examples are given of processes without dominating measures that satisfy the ratio condition.  For the first, consider the measure on $\mathbb{N}^{\mathbb{Z}}$, given by setting, for each $n \geq 1$,
\[
\mathbb{P}\left(\ldots, X_{-1}=n, X_0=n, X_{1}=n, X_2 = n, \ldots\right) = \frac{1}{2^n}.
\]
This certainly satisfies the ratio condition and has no finite dominating measure.
\qed \end{example}

The next example is a process that satisfies the ratio condition without having dominating measure that is also ergodic and mixing.  

\begin{example}
Consider again a measure on doubly infinite words on $\mathbb{N}$, with the measure on $X_0$ given in terms of the values of $X_{-1}$ and $X_{-2}$.  Define
\begin{equation}\label{eq:2-step-Markov}
\mathbb{P}\left(X_0=c \mid X_{-1} = a, X_{-2}=b\right)=
	\begin{cases}
		\frac{1}{2}	&\text{ if } a \neq b \text{ and } c=a,\\
		\frac{2^{-c-1}}{1-2^{-a}}		&\text{ if } a \neq b \text{ and } c \neq a,\\
		0			&\text{ if } a=b=c, \text{ and }\\
		\frac{2^{-c}}{1-2^{-b}}	&\text{ if } a=b \text{ and } c \neq b.
	\end{cases}
\end{equation}
At stationary process satisfying equation \eqref{eq:2-step-Markov} is a 2-step Markov process with stationary measure on pairs satisfying
\begin{equation}
\mathbb{P}\left(X_0 = a, X_1 = b \right) = 
	\begin{cases}
		\frac{3}{4} \cdot \frac{1}{2^{a+b}}	&\text{ if } a \neq b\\
		\frac{3}{4} \cdot \frac{1}{2^a}\left(1 - \frac{1}{2^a}\right)	&\text{ if } a = b.
	\end{cases}
\end{equation}
Thus, the stationary distribution on any particular coordinate is
\[
\mathbb{P}\left(X_0 = a\right) = \frac{3}{2}\frac{1}{2^a}\left(1-\frac{1}{2^a}\right).
\]
This process satisfies the ratio condition, Definition \ref{cond:four}, as it is a two-step Markov chain.  While $\left(X_i\right)_{i \in \mathbb{Z}}$ is a uniform martingale, this process does not have a finite dominating measure as, for example, a dominating measure would have to give measure at least $\frac{1}{2}$ to each integer.
\qed \end{example}

These examples show that, in terms of the conditions considered here, the implication in the following theorem is as strong as possible.  This proof is closely related to that given by Kalikow~\cite{Kalikow90} for $2$-letter alphabets.  Recall Theorem \ref{thm:ratio}:

\begin{thmrat}
Let $A$ be any set and $\Omega = A^{\mathbb{Z}}$.  Let $\mathbb{P}$ be a stationary probability measure on $\Omega$ so that $\left(\Omega, \mathbb{P}\right)$ satisfies the ratio condition as in Definition \ref{cond:four}, then $\left(\Omega, \mathbb{P}\right)$ is a random Markov process.
\end{thmrat}

\begin{proof}
Let $\left(p_i\right)_{i\geq 1} \subseteq \left(0,1\right)$ with $\sum_{i=1}^{\infty} p_i = 1$.  It is shown in this proof that there is a sequence $\left\{n_i\right\}_{i \geq 1}$ and a process $\left(L_n\right)_{n \in \mathbb{Z}}$ with a coupling $\hat{\mathbb{P}}$ such that $\left(\left(X_n, L_n\right)_{n \in \mathbb{Z}}, \hat{\mathbb{P}}\right)$ is a random Markov process and for each $i \geq 1$, $\hat{\mathbb{P}}\left(L_0 = n_i\right) = p_i$.

The sequence $n_i$ is chosen recursively.  Using the ratio condition, as in Definition \ref{cond:four}, choose $n_1$ large enough so that for all $n \geq n_1$, and for all $\mathbf{\omega}$, $E$,
\[
\left|\frac{\mathbb{P}\left(X_0 \in E \mid X_{-1} = \omega_{-1}, \ldots\right)}{\mathbb{P}\left(X_0 \in E \mid X_{-1} = \omega_{-1}, \ldots, X_{-n} = \omega_{-n}\right)}  - 1\right| < \frac{p_1}{2}.
\]

For all $i \geq 1$, given $n_i$, choose $n_{i+1} \geq n_i$ to be such that for all $n \geq n_{i+1}$, for all $\mathbf{\omega}$, $E$
\[
\left|\frac{\mathbb{P}\left(X_0 \in E \mid X_{-1} = \omega_{-1}, \ldots\right)}{\mathbb{P}\left(X_0 \in E \mid X_{-1} = \omega_{-1}, \ldots, X_{-n} = \omega_{-n}\right)}  - 1\right| < \frac{p_{i+1}}{2\sum_{j \leq i+1} p_j}.
\]

For every $i \geq 1$, $\mathbf{\omega}$, define a measure, $\mu_{\mathbf{\omega}, i}$ on $A$ by
\[
\mu_{\mathbf{\omega}, i}\left(E\right) = \mathbb{P}\left(X_0 \in E \mid X_{-1} = \omega_{-1}, \ldots, X_{-n_i}= \omega_{-n_i}\right).
\]
As in the proof of Theorem \ref{thm:countable}, the goal is to show that there is a joint distribution with a process $\left(L_n\right)_{n \in \mathbb{Z}}$ so that 
\[
\mu_{\mathbf{\omega}, i}\left(E\right) = \hat{\mathbb{P}}\left(X_0 \in E \mid X_{-1} = \omega_{-1}, \ldots, X_{-n} = \omega_{-n} \wedge L_0 \leq n_i\right).
\]
  For each $i \geq 1$, define another measure $\tau_{\mathbf{\omega}, i}$ as follows.  For $i=1$, set $\tau_{\mathbf{\omega}, 1} = \mu_{\mathbf{\omega}, 1}$ and for $i \geq 1$ define 
\begin{equation}\label{eq:uncountable_table}
\tau_{\mathbf{\omega}, i+1}\left(E\right) = \frac{1}{p_{i+1}}\left(\left(\sum_{j \leq i+1} p_j\right)\mu_{\mathbf{\omega}, i+1}\left(E\right) - \left(\sum_{j \leq i} p_j\right)\mu_{\mathbf{\omega}, i}\left(E\right) \right).
\end{equation}

Since $\mu_{\mathbf{\omega}, i}$ and $\mu_{\mathbf{\omega}, i+1}$ are both probability measures on $A$, $\tau_{\mathbf{\omega}, i+1}$ is a signed measure on $A$ with $\tau_{\mathbf{\omega}, i}\left(A\right) = 1$.  Note also, that by definition, for every measurable set $E$ and every $i$, the function $\mathbf{\omega} \mapsto \mu_{\mathbf{\omega}, i}(E)$ is a $\mathbb{P}$-measurable function of $\mathbf{\omega}$.  Thus, the function $\mathbf{\omega} \mapsto \tau_{\mathbf{\omega}, i}(E)$ is also $\mathbb{P}$-measurable. To show that $\tau_{\mathbf{\omega}, i+1}$ is also a positive measure, note that for every event $E$,
\begin{align*}
\frac{\mu_{\mathbf{\omega}, i+1}\left(E\right)}{\mu_{\mathbf{\omega}, i}\left(E\right)}
	&=\frac{\mathbb{P}\left(X_0 \in E \mid X_{-1} = \omega_{-1}, \ldots, X_{-n_{i+1}} = \omega_{-n_{i+1}}\right)}{\mathbb{P}\left(X_0 \in E \mid X_{-1} = \omega_{-1}, \ldots, X_{-n_i} = \omega_{-n_i}\right)}\\
	&=\frac{\mathbb{P}\left(X_0 \in E \mid X_{-1} = \omega_{-1}, \ldots, X_{-n_{i+1}} = \omega_{-n_{i+1}}\right)}{\mathbb{P}\left(X_0 \in E \mid X_{-1} = \omega_{-1}, \ldots\right)} \\
	&\qquad \cdot \frac{\mathbb{P}\left(X_0 \in E \mid X_{-1} = \omega_{-1}, \ldots\right)}{\mathbb{P}\left(X_0 \in E \mid X_{-1} = \omega_{-1}, \ldots, X_{-n_i} = \omega_{-n_i}\right)}\\
	&\geq \left(1-\frac{p_{i+1}}{2\sum_{j \leq i+1} p_j}\right)\left(1+\frac{p_{i+1}}{2\sum_{j \leq i+1}p_j}\right)^{-1}\\
	&\geq 1- \frac{p_{i+1}}{\sum_{j \leq i+1}p_j} = \frac{\sum_{j \leq i} p_j}{\sum_{j \leq i+1} p_j}.
\end{align*}

Thus, $\left(\sum_{j \leq i+1} p_j\right)\mu_{\mathbf{\omega}, i+1}\left(E\right) \geq \left(\sum_{j \leq i} p_j\right)\mu_{\mathbf{\omega}, i}\left(E\right)$ which implies that for every event $E$, $\tau_{\mathbf{\omega}, i+1}\left(E\right) \geq 0$.

The measure $\hat{\mathbb{P}}$ is defined so that for each $i$, $\mathbf{\omega}$ and $E$.
\begin{equation}\label{eq:unctbl_measure}
\hat{\mathbb{P}}\left(X_0 \in E \mid X_{-1} = \omega_{-1}, \ldots, X_{-n_i} = \omega_{-n_i} \wedge L_0 = n_i\right) = \tau_{\mathbf{\omega}, i}\left(E\right).
\end{equation}
Equation \eqref{eq:unctbl_measure} can be used together with the original measure $\mathbb{P}$ and the fact that $\tau_{\mathbf{\omega}, i}(E)$ is a $\mathbb{P}$-measurable function of $\mathbf{\omega}$, to define a probability measure, $\hat{\mathbb{P}}$ on the doubly infinite sequences $\left(X_n, L_n\right)_{n \in \mathbb{Z}}$.

\end{proof}

\section{Conclusion}\label{sec:open}

Many new problems have been raised about extending results about random Markov process on finite alphabets to processes on infinite alphabets.  The question of whether or not a uniform martingale has a representation as a random Markov process with finite expected look-back distance is of interest in light of the result of Kalikow~\cite{Kalikow90} regarding sufficient conditions for a process to be weak-Bernoulli.  Kalikow~\cite{Kalikow90} showed that a finite state random Markov process with finite expected look-back distance and satisfying some other conditions such as being weak-mixing or having some state whose probability given any past is bounded away from $0$, is weak-Bernoulli.  While this result~\cite[Theorem 7]{Kalikow90} was stated only for processes on finite alphabets, the proof remains valid for countable alphabets.   

In general, it remains unknown whether every uniform martingale on a countable alphabet with a finite dominating measure and $\sum_n \var_n\left(\mathbb{P}\right) < \infty$ has a representation as a random Markov process with finite expected look-back distance.  If this is not true, there are some special cases raised by previous work on random Markov processes that remain of interest.

One area in which open questions remain are the connections between uniform martingales, random Markov processes and extension of processes, as in the sense used in isomorphism ergodic theory.  Recall that the \emph{shift map} $T$ on doubly infinite words is the function $(\omega_i)_{i \in \mathbb{Z}} \mapsto (\omega_{i+1})_{i \in \mathbb{Z}}$.  A stationary process $(X_i)_{i \in \mathbb{Z}}$ on an alphabet $A$ is said to \emph{extend to} a stationary process $(Y_i)_{i \in \mathbb{Z}}$ on an alphabet $A'$ if{f} there is a function $f: A^{\mathbb{Z}} \to A'^{\mathbb{Z}}$ that is measurable, measure preserving and commutes with the shift operator.


Kalikow, Katznelson, and Weiss~\cite{KKW92} showed that every zero-entropy process can be extended to a uniform martingale on a finite alphabet.  One could ask whether every zero-entropy process be extended to a uniform martingale on a countable alphabet with $\sum_n \var_n < \infty$?

In~\cite{Kalikow12}, Kalikow proved that every process can be extended to a uniform martingale, and in fact to a random Markov process. In this paper we have established that it is of particular interest when such a process has a finite dominating measure for the present given the past. A further direction would be to establish necessary and sufficient conditions for a process to be extendable to a random Markov process with a finite dominating measure.  In particular, since a process has finite entropy if{f} it can be displayed as process with a finite alphabet, one interesting question is whether every process with finite entropy can be extended to a random Markov process with a finite alphabet.  This question was previously posed by Kalikow in~\cite{Kalikow90}.

In this paper and in a previous paper by Kalikow~\cite{Kalikow12}, processes with a countable alphabet have been studied in terms of the  categories in which such processes can be displayed. Here, we have also displayed finite processes as deterministic random Markov processes and looked at categories of processes with an uncountable alphabet. Isomorphism ergodic theory was launched by Don Ornstein with his isomorphism theorem and since then a great many theorems about isomorphism classes have been proved and extended to processes with a countable and uncountable alphabets. Now that we have displayed these additional ways of looking at such processes, it is our hope that future work will extend isomorphism ergodic theory to study these classifications also.

\section*{Acknowledgements}

The authors wish to thank Paul Balister for sharing his construction for the sets of integers in the special case of two-letter alphabets presented in Section \ref{subsec:determ_construction}.  It was his construction and our generalization of it to arbitrary finite alphabets that set us in the direction of considering deterministic random-step Markov processes with finite expected look-back distance.

\bibliographystyle{amsplain}
\bibliography{ErgodicBib}

\end{document}